\documentclass[reqno,11pt]{amsart}
\usepackage{a4wide,color,eucal,enumerate,mathrsfs}
\usepackage[normalem]{ulem}
\usepackage{amsmath,amssymb,epsfig,amsthm,scalerel} %,bbm
\usepackage[latin1]{inputenc}
\usepackage{psfrag}

\def\az{\alpha}

\def\ez{\epsilon}

\def\bint{{\ifinner\rlap{\bf\kern.35em--}
\int\else\rlap{\bf\kern.45em--}\int\fi}\ignorespaces}

\def\bbint{{\ifinner\rlap{\bf\kern.35em--}
\hspace{0.078cm}\int\else\rlap{\bf\kern.45em--}\int\fi}\ignorespaces}

\newcommand{\R}{\mathbb{R}}

\newtheorem{thm}{Theorem}[section]
\newtheorem{lem}[thm]{Lemma}%[section]     %@@!!@@!!
\newtheorem{prop}[thm]{Proposition}%[section]    %@@!!@@!!
\newtheorem{cor}[thm]{Corollary}%[section]    %@@!!@@!!
\numberwithin{equation}{section}

\theoremstyle{remark}
\newtheorem{rem}[thm]{Remark}%[section]    %@@!!@@!!

\def\bint{{\ifinner\rlap{\bf\kern.35em--}
\int\else\rlap{\bf\kern.45em--}\int\fi}\ignorespaces}

\usepackage{graphicx}
\usepackage{import}
\usepackage{xifthen}
\usepackage{pdfpages}
\usepackage{transparent}

\title{Sharp gradient stability for the Sobolev inequality}
\author{Alessio Figalli, Yi Ru-Ya Zhang}
\date{\today}
\address{ETH Z\"urich, Department of Mathematics, R\"amistrasse 101, 8092, Z\"urich, Switzerland}
\email{alessio.figalli@math.ethz.ch}  
\email{yizhang3@ethz.ch}


\begin{document}

\begin{abstract}
We prove a sharp quantitative version of the
$p$-Sobolev inequality for any $1<p<n$, with a control on the strongest possible 
distance from the class of optimal functions. Surprisingly, the sharp exponent is constant for $p \leq 2$, while it depends on $p$ for $p>2$.
 \end{abstract}

\maketitle

\section{Introduction}

Motivated by important applications to problems in the calculus of variations and evolution PDEs, in recent years there has been a growing 
interest around the understanding of quantitative stability for functional/geometric inequalities, see for instance \cite{BE1991,BWW2003, C2006,FMP2007,FMP2008, FMP2009,CFMP2009, FMP2010, GW2010, FM2011, CL2012,CF2013,CFW2013, CS2013, DT2013,FM2013,FMP2013, R2014, F2015,C2017,DZ2017,FJ2017A, FJ2017C,FMM2018, FZ2019, HST20191, HST20192, FN2019, N2019, N2020},
as well as the survey papers \cite{F2013, F2015, FJ2017C}.  Following this line of research, in this paper we shall investigate the stability of minimizers to the classical Sobolev inequality.

\subsection{The Sobolev inequality}
The question of quantitative stability for the Sobolev inequality was first raised by Brezis and Lieb \cite{BL1985}.
Before describing the problem and the state of the art, we first introduce some useful definitions.

Given $n\ge 2$ and $1<p<n$, denote by $\dot W^{1,p}(\mathbb R^n)$ the closure of  $C^\infty_c(\R^n)$ with respect to the norm
$$\|u\|_{\dot{W}^{1,p}(\mathbb R^n)}=\left(\int_{\mathbb R^n} |Du|^p \, dx\right)^{\frac 1 p}.$$
The Sobolev inequality guarantees the existence of a positive constant $S=S(n,p)$ such that
$$ {\|Du\|_{L^p(\mathbb R^n)}}\ge S{\|u\|_{L^{p^*}(\mathbb R^n)}},$$
where $p^*=\frac {np} {n-p}. $
We call the largest constant $S$ satisfying this property the {\it optimal Sobolev constant}. 

Let $\mathcal M$ be the $(n+2)$-dimensional manifold of all functions of the form
$$v_{a,b,x_0}(x):=\frac {a}{\left(1+b|x-x_0|^{\frac{p}{p-1}}\right)^{\frac{n-p}{p}}},\qquad a\in \R\setminus \{0\},\,b>0,\,x_0\in \mathbb R^n.$$
As shown in  \cite{A1976, T1976,CNV2004}, the set $\mathcal M$ corresponds to the set of all weak solutions to 
\begin{equation}\label{S equ}
-\Delta_{p}v=S^p \|v\|^{p-p^*}_{L^{p^*}(\mathbb R^n)} |v|^{p^*-2}v,
\end{equation}
where $S$ is the optimal Sobolev constant and 
$$-\Delta_p v= {\rm div}(|Dv|^{p-2}Dv).$$
It is also proven that $\mathcal M$ coincides with the set of all the extremal functions in the Sobolev inequality; in particular,
$$ {\|Dv\|_{L^p(\mathbb R^n)}}= S{\|v\|_{L^{p^*}(\mathbb R^n)}}\qquad \forall\,v \in \mathcal M.$$

\subsection{The stability question: the generalized Brezis-Lieb's problem}
To formulate our stability problem, we introduce the notion of {\it $p$-Sobolev deficit}:
\begin{equation}
\label{eq:deficit}
\delta(u):=\frac{\|Du\|_{L^p(\mathbb R^n)}}{\|u\|_{L^{p^*}(\mathbb R^n)}}- S \qquad \forall\,u \in \dot W^{1,p}(\R^n).
\end{equation}
Note that $\delta \geq 0$, and it vanishes only on $\mathcal M$.

\smallskip

In \cite{BL1985},
Brezis and Lieb asked whether, for $p=2$, the deficit can be estimated from below by some appropriate distance between $u$ and $\mathcal M$, together with a suitable decay.
This problem was settled few years later by  Bianchi and Egnell \cite{BE1991}: they showed the existence of a constant $c=c(n)>0$ such that
$$
\delta(u)\geq c\inf_{v\in \mathcal M}\biggl(\frac{\|Du-Dv\|_{L^{2}(\mathbb R^n)} }{\| Du\|_{L^{2}(\mathbb R^n)}}\biggr)^{2} \qquad \forall\,u \in \dot W^{1,2}(\R^n),
$$
which is optimal both in terms of the strength of the distance from $\mathcal M$, and in terms of the exponent $2$ appearing in the right hand side.

After this work, it became immediately of interest understanding whether Brezis-Lieb's question could be solved also for general values of $p$.
Unfortunately, Bianchi-Egnell's method heavily depended on the Hilbert structure of $\dot {W}^{1,\,2}(\R^n)$, so new ideas and techniques were needed.

\smallskip

Almost 20 years later, in \cite{CFMP2009}, Cianchi, Fusco, Maggi, and Pratelli proved a stability version for every $p \in (1,n)$ with distance given by
	\begin{equation}
	\label{aim weak}\inf_{v\in \mathcal M}\biggl(\frac{\|u-v\|_{L^{p^*}(\mathbb R^n)} }{\|u\|_{L^{p^*}(\mathbb R^n)}}\biggr)^{\alpha}  \qquad \forall\,u \in \dot W^{1,p}(\R^n),
	\end{equation}
	together with the explicit decay exponent $\alpha=\alpha_{\scaleto{CFMP}{3.5pt}}:=\left[p^*\left(3+4p-\frac {3p+1}{n}\right)\right]^2$. Although most likely the result was not sharp, this was the first stability result valid for the full range of $p$.
In addition, their proof introduced in this problem a beautiful combination of techniques coming from symmetrization theory and  optimal transport.

These technique were further developed by Figalli, Maggi, and Pratelli in \cite{FMP2013} to provide a sharp stability result ---both in terms of the notion of distance and of the decay exponent--- in the special case $p=1$ (for this case, see also the earlier results \cite{C2006, F2015, FMP2007}).

Still, until few years ago, it remained a major open problem whether the $p$-Sobolev deficit could control the closeness to $\mathcal M$ at the level of the gradients (i.e., the strongest distance that one may hope to control with $\delta(u)$), as in the case of Bianchi and Egnell.

A first answer to this question was given by Figalli and Neumayer in \cite{FN2019} in the case $p>2$,
where they developed in $\dot W^{1,p}(\R^n)$ a suitable analogue of the strategy in \cite{BE1991} to prove the existence of a constant $c=c(n,p)>0$ such that
	\begin{equation}
	\label{aim}
	\delta(u)\ge c\inf_{v\in \mathcal M}\biggl(\frac{\|D(u-v)\|_{L^p(\mathbb R^n)} }{\|Du\|_{L^{p}(\mathbb R^n)}}\biggr)^{\az} \qquad \forall\,u \in \dot W^{1,p}(\R^n)
	\end{equation}
	where $\az = p\alpha_{\scaleto{\scaleto{CFMP}{3.5pt}}{3.5pt}}$, with $\alpha_{\scaleto{CFMP}{3.5pt}}$ as above. The appearance of the exponent $\alpha_{\scaleto{\scaleto{CFMP}{3.5pt}}{3.5pt}}$ comes from the fact that, in one of the steps in the proof, the authors need to rely on the result in \cite{CFMP2009}.
	
	Very recently, in \cite{N2020}, Neumayer extended \eqref{aim} to the full range $1<p<n$. While her proof is much simpler than the one in \cite{FN2019}, it relies heavily on the result in \cite{CFMP2009} and her strategy cannot give the sharp exponent in \eqref{aim}, even if one could prove \eqref{aim weak} with a sharp exponent.
	In particular, her approach provides the same exponent as the one in \cite{FN2019} when $p>2$, while it gives \eqref{aim} with $\az = \frac{p}{p-1}\alpha_{\scaleto{CFMP}{3.5pt}}$ when $p \in (1,2)$.

Despite all these developments, the stability exponent appearing in all these previous results was far from optimal.
The aim of this paper is to give a final answer to this problem by proving \eqref{aim} for all $1<p<n$ with sharp exponent.
	
	Here is our theorem:
	\begin{thm}\label{main thm}
	Let $1<p<n$, and define $\delta(\,\cdot\,)$ as in \eqref{eq:deficit}.
	There exists a constant $c=c(n,p)>0$ such that \eqref{aim} holds  with $\az=\max\{2,\,p\}$.
\end{thm}

\begin{rem} The decay exponent $\az=\max\{2,\,p\}$ is sharp, as we now explain.

Fix $v=v_{1,1,0} \in \mathcal M$ and consider first $u_i:=v(A_ix)$, where $A_i \in \R^{n\times n}$ denotes the diagonal matrix 
$$A_i={\rm diag}\left(1, \dots, 1, 1+\frac 1 i\right).$$ 
It is not difficult to check that $\delta(u_i)$ behaves as $i^{-2}$, while the right hand side of \eqref{aim} behaves as $i^{-\alpha},$ 
hence \eqref{aim} cannot hold with $\az<2$. 

On the other hand, fix $\varphi \in C_c^\infty(B_1)$ a nontrivial function, and consider now
 $\tilde u_i:=v+  \varphi(x_i+\cdot)$,  where
$x_i\in \mathbb R^n$ is a sequence of points converging towards $\infty$.
One can check that
$$\|D\tilde u_i\|^p_{L^p(\mathbb R^n)}= \|Dv\|^p_{L^p(\mathbb R^n)} +\|D\varphi\|^p_{L^p(\mathbb R^n)}+r_{i,1}$$
and
$$\|\tilde u_i\|^{p^*}_{L^{p^*}(\mathbb R^n)}= \|v\|^{p^*}_{L^{p^*}(\mathbb R^n)} +\|\varphi\|^{p^*}_{L^{p^*}(\mathbb R^n)}+r_{i,2},$$
with $|r_{i,1}|+|r_{i,2}|\leq C\bigl(v(x_i)+|Dv(x_i)|\bigr)\leq Cv(x_i) \to 0$ as $i\to \infty$.
Hence, choosing a sequence $\ez_i\to 0$ such that $v(x_i)\ll \ez_i\ll 1$, the functions $\hat u_i:=v+ \ez_i \varphi(x_i+\cdot)$ satisfy
$$\|D\hat u_i\|^p_{L^p(\mathbb R^n)}= \|Dv\|^p_{L^p(\mathbb R^n)} +\ez_i^p\|D\varphi\|^p_{L^p(\mathbb R^n)}+o(\ez_i^p)$$
and 
$$\|\hat u_i\|^{p^*}_{L^{p^*}(\mathbb R^n)}= \|v\|^{p^*}_{L^{p^*}(\mathbb R^n)} +\ez_i^{p^*}\|\varphi\|^{p^*}_{L^{p^*}(\mathbb R^n)}+o(\ez_i^{p^*}).$$
Thanks to these facts, one easily deduces that $\delta(\hat u_i)$ behaves as $\ez_i^{p}$, while the right hand side of \eqref{aim} behaves as $\ez_i^{\alpha}$. Thus \eqref{aim} cannot hold with $\az<p$.
\end{rem}

\subsection{Strategy of the proof}
\label{sect:idea}
As in \cite{FN2019}, the starting idea
comes from \cite{BE1991}.

More precisely, given $u$ close to $\mathcal M$, one chooses  $v \in \mathcal M$ close to $u$ and set $\varphi:=\frac{u-v}{\|\nabla u-\nabla v\|_{L^p(\R^n)}}$ and $\ez:=\|\nabla u-\nabla v\|_{L^p(\R^n)}$,
so that $u$ can be written as $v+\ez \varphi$.
Then one expands $\delta(u)$ in $\epsilon$, and one aims to use it to control $\nabla \varphi$ in $L^p$.

\smallskip

When $p=2$, as shown in \cite{BE1991}, the expansion of $\delta(u)$ gives 
$$
\delta(v+\ez \varphi)=\ez^2Q_v[\varphi]+o\bigl(\epsilon^2\|D\varphi\|_{L^2(\R^n)}^2\bigr), 
$$
where $Q_v[\,\cdot\,]$ is a quadratic form depending on $v$. In addition, if $\varphi$ is orthogonal to $T_v\mathcal M$ in the weighted  space $L^2(\R^n;v^{2^*-2})$, spectral analysis shows that $Q_v[\varphi]$ controls $\|D\varphi\|^2_{L^2}$  from above,
thus
$$
\delta(v+\ez \varphi)\geq c\ez^2\|D\varphi\|_{L^2(\R^n)}^2+o\bigl(\epsilon^2\|D\varphi\|_{L^2(\R^n)}^2\bigr).
$$
Hence the result follows for $\ez \ll 1$, provided orthogonality can be ensured. In the case $p=2$, this can be easily guaranteed by choosing $v$ which minimizes $$
\mathcal M\ni v\mapsto \|\nabla u-\nabla v\|_{L^2(\R^n)},
$$ 
completing the proof.

\smallskip

For $p>2$, in \cite{FN2019} the authors tried to mimic the strategy of \cite{BE1991}. More precisely,  the expansion of $\delta(u)$ gives
$$
\delta(v+\ez \varphi)=\ez^2Q_v[\varphi]+o\bigl(\epsilon^2\|D\varphi\|_{L^p(\R^n)}^2\bigr), 
$$
where $Q_v[\,\cdot\,]$ is a quadratic form depending on $v$ and $p$.
Again, if $\varphi$ is orthogonal to $T_v\mathcal M$ in the weighted  space $L^2(\R^n;v^{p^*-2})$, spectral analysis shows that $Q_v[\varphi]$ controls the weighted norm $\|D\varphi\|^2_{L^2(\R^n;|Dv|^{p-2})}$ from above,
thus
$$
\delta(v+\ez \varphi)\geq c\ez^2\|D\varphi\|_{L^2(\R^n;|Dv|^{p-2})}+o\bigl(\epsilon^2\|D\varphi\|_{L^p(\R^n)}^2\bigr). 
$$
Unfortunately, now this argument is not sufficient, since for $p>2$ the $L^p$ norm of $D\varphi$ may not be controllable by its weighted $L^2$ norm. 
Furthermore, finding the correct orthogonality condition in this non-Hilbertian context requires a series of new ideas. All this creates a series of challenges that were overcome in \cite{FN2019} by relying also on the $L^{p^*}$ stability result  of \cite{CFMP2009}, as explained in detail in \cite[Section~2]{FN2019}.

\smallskip

In this paper, to handle the general case $1<p<n$ and prove a stability estimate with sharp exponent, we need to face several new difficulties.
The idea is again to expand the deficit $\delta(v+\ez \varphi)$. However, the argument in \cite{FN2019} shows that, for $p\neq 2$, a standard Taylor expansion creates error terms that cannot be controlled. Even worse, a second order expansion of the deficit naturally leads to a quadratic form consisting of a weighted $\dot W^{1,2}$  and a weighted $L^{2}$ norm. However,  when $p<2$, the $\dot W^{1,p}$ norm is weaker than any weighted $\dot W^{1,2}$ norm, so we cannot expand the deficit at order $2$ (this was the main reason why \cite{FN2019} could only deal with the case $p\geq 2$). In addition, when $p\leq \frac{2n}{n+2}$ (equivalently $p^*\leq 2$), the $L^{p^*}$ norm is not sufficient to control any weighted $L^2$ norms, and this creates even further challenges.
For all these reasons, our arguments are different in the three regimes $p \in (1,\frac{2n}{n+2}]$, $p \in (\frac{2n}{n+2},2)$, and $p \in [2,n).$

To briefly explain the main ideas in the proof, let us focus on the case $p \in (1,\frac{2n}{n+2}]$ (note that this set is nonempty only for $n\geq 3$).
As mentioned above, a first problem consists in understanding  how to expand the deficit. With no loss of generality, we can assume that $v>0$.

Our first new tool is provided by the following inequalities: for any $\kappa>0$ there exists $C_1>0$ such that, for $\ez$  sufficiently small,
\begin{multline*}
\|Dv+\ez D\varphi\|^p_{L^p(\mathbb R^n)}\ge  \int_{\mathbb R^n} |Dv|^p\,dx+\ez p \int_{\mathbb R^n} |Dv|^{p-2}Dv\cdot D\varphi \,dx  \\
  +    \frac{\ez^2p(1-\kappa)} 2  \biggl(\int_{\mathbb R^n} |Dv|^{p-2}|D\varphi|^2 + (p-2) |w|^{p-2}\biggl(\frac{|Du|-|Dv|}{\ez}\biggr)^2 \,dx \biggr)
\end{multline*}
and
\begin{multline*}
\|v+\ez\varphi\|^p_{L^{p^*}(\mathbb R^n)}\le  \|v\|^p_{L^{p^*}(\mathbb R^n)} \\
+ \|v\|_{L^{p^*}(\mathbb R^n)}^{p-p^*}\left(\ez p \int_{\mathbb R^n} v^{p^*-1}\varphi\,dx +\ez^2\left(\frac{ p(p^*-1)} 2 +\frac{p \kappa}{p^*}\right)\int_{\mathbb R^n}\frac {(v  +C_1|\ez\varphi|)^{p^*}}{v^2+|\ez \varphi|^2}|\varphi|^2\,dx\right),
\end{multline*}
where $w=w(Dv,Du)$ is obtained by taking a suitable combination of $Dv$ and $Du$ (depending on their respective sizes) as in Lemma \ref{vector inequ}.

Combining these inequalities and using \eqref{S equ}, one gets
\begin{multline}
\label{eq:expand delta 1}
C(n,p) \delta(u) \geq  \frac{\ez^2 p(1-\kappa) } 2  \biggl(\int_{\mathbb R^n} |Dv|^{p-2}|D\varphi|^2 + (p-2) |w|^{p-2}\biggl(\frac{|Du|-|Dv|}{\ez}\biggr)^2 \,dx \biggr)  \\
 - \ez^2 \|v\|_{L^{p^*}(\mathbb R^n)}^{p-p^*} S^p\left(\frac{p (p^*-1)} 2 +\frac{p \kappa}{p^*}\right)\int_{\mathbb R^n}\frac {(v  +C_1|\ez\varphi|)^{p^*}}{v^2+|\ez\varphi|^2}|\varphi|^2\,dx,
\end{multline}
so the result would be proved if we could show that, under some suitable orthogonality relation between $v$ and $\varphi$, the right hand side above controls $\|\ez D\varphi\|_{L^p(\R^n)}^{\max\{2,p\}}$ for $\ez \ll 1$.
Unfortunately this is false for $p<2$, since 
$$
\ez^2|Dv|^{p-2}| D\varphi|^2 + (p-2) |w|^{p-2}\bigl({|Du|-|Dv|}\bigr)^2\sim \ez |Dv|^{p-1}| D\varphi| \qquad \text{for }|Dv|\leq \ez|D\varphi|
$$
(cp. \eqref{eq:lower G}), and in general this weighted $W^{1,1}$ norm of $\varphi$ is not enough to control the last term in~\eqref{eq:expand delta 1}.

Hence, our second goal consists in showing that we can improve the expansion  of $\|Dv+\ez D\varphi\|^p_{L^p(\mathbb R^n)}$
(see Lemma \ref{vector inequ}), so that we can add the extra term 
$c_0 \int_{\mathbb R^n}  \min\bigl\{ \ez^p|D\varphi|^p,\,\ez^2|Dv|^{p-2}  |D\varphi|^2\bigr\} \,dx$ to the  right hand side of \eqref{eq:expand delta 1}.
With this extra term at our disposal, we now want to use the right hand side of \eqref{eq:expand delta 1} to control $\|\ez D\varphi\|_{L^p(\R^n)}^{\max\{2,p\}}$.

The main idea behind the proof of this fact consists of two steps:\\ (1) show that the result is true if one replaces the two integrands in the right hand side of \eqref{eq:expand delta 1} by their limit as $\ez \to 0$;\\
(2) show by compactness that the result holds also for $\ez$ sufficiently small.\\
Thanks to the spectral analysis performed in \cite{FN2019}, Step (1) is rather easy, as it boils down to proving a compact embedding (see Propositions~\ref{compact embedding} and \ref{spectrum gap}).
On the other hand, Step (2) turns out to be extremely delicate. A key difficulty comes from the fact that the integrand appearing in the last term of \eqref{eq:expand delta 1} behaves 
like $v^{p^*-2}|\varphi|^2$ when $|\varphi|\ll \frac{v}{\ez}$, and like $\ez^{p^*-2}|\varphi|^{p^*}$ otherwise.
Analogously, the first integrand  behaves like $|Dv|^{p-2}|D\varphi|^2$ when $|D\varphi|\ll \frac{|Dv|}{\ez}$, and like $\ez^{p-2}|D\varphi|^{p}$ otherwise.
These substantial changes of behavior, and the fact that a change in size of the gradients does not necessarily correspond to a change in size of the functions, make the proofs of several results (in particular the ones of Lemma~\ref{weighted compact} 
and Proposition~\ref{new spectrum gap}) very involved. 

Finally, once all these difficulties have been solved, in Section~\ref{sect:pf thm} we introduce a new minimization principle to select $v$ so to guarantee orthogonality and conclude the proof.

\subsection{Structure of the paper}
In Section \ref{sect:vector}, we prove a series of new vectorial inequalities
that play a crucial role in the expansion of the deficit. 
In Section~\ref{spectrum gaps}, we prove the compactness and spectral gaps estimates required for the proof of Theorem~\ref{main thm}, which is then postponed to Section~\ref{sect:pf thm}.
Finally, we collect some technical estimates in two appendices.

 \medskip
  
{\noindent \bf Notation.}  In our estimates we often write 
constants as 
positive real numbers $C(\cdot)$ and $c(\cdot)$, with
the parentheses including all the parameters on which the constant depends. Usually we use $C$ to denote a constant larger than $1$, and $c$ for a constant less than $1$.
We simply write $C$ or $c$ if the constant is absolute. The constant $C(\cdot)$ may
vary between appearances, even within a chain of inequalities. 
The notation $a \sim b$ indicates that both inequalities $a\le C b$ and $b\le C a$ hold. We denote the closure of a   set  $A\subset \R^n$ by $\overline{A}$. Finally, the Euclidean ball centered at $x$ with radius $r$ is denoted by $B(x,r)$.

\medskip

{\noindent \it Acknowledgments.}
The second author would like to thank Herbert Koch for several discussions about this problem during his stay at the Hausdorff Center for Mathematics in Bonn.
Both authors are grateful to Federico Glaudo and Robin Neumayer for useful comments on a preliminary version of this manuscript.
Both authors are supported by the European Research Council under the Grant Agreement No.
721675 ``Regularity and Stability in Partial Differential Equations (RSPDE).''

\section{Sharp vector inequalities in Euclidean spaces} 
\label{sect:vector}

We start with the following sharp inequalities on vectors, which improve the ones in \cite[Section~3.2]{FN2019}.
The basic idea behind these inequalities is the following: to apply the strategy described in Section~\ref{sect:idea}, for  fixed $x \in \R^n$ we would like to find a non-negative quadratic expression in $y$ that controls $|x+y|^p-|x|^p+p|x|^{p-2}x\cdot y$ from below, and that for $|y|\ll 1$ behaves like the Hessian of $z\mapsto |z|^p$ at $x$ (this is needed in order to exploit later Proposition~\ref{spectrum gap}). Unfortunately this is impossible, so we introduce a weight $|w|=|w(x,x+y)|$ that depends on the sizes of $|x|$ and $|x+y|$ and modulates the quadratic-type expression appearing in the right hand side of our estimates. 
Analogously, in Lemma~\ref{upper bound}(i) we need to consider a weighted expression in front of $|b|^2$ in order to obtain a sufficiently precise expansion. We note that, as explained in Section~\ref{sect:idea}, the extra term (the one multiplied by  $c_0$) appearing in Lemma~\ref{vector inequ} will be crucial to prove our main theorem.

\begin{lem}\label{vector inequ}
Let $x,y \in \mathbb R^n$. Then, for any $\kappa>0$, there exists a constant $c_0=c_0(p,\kappa)>0$ such that the following holds:

\noindent(i) For $1<p<2$,
\begin{multline*}
|x+y|^p\ge |x|^p+p|x|^{p-2}x\cdot y+\frac {1-\kappa} 2\left( p |x|^{p-2} |y|^2+ p(p-2)|w|^{p-2}\bigl( |x|-|x+y| \bigr)^2\right)\\+c_0\min \left\{|y|^p,\,|x|^{p-2}|y|^2\right\},
\end{multline*}
where 
$$
w=w(x,x+y) := \left\{ \begin{array}{cl}
	 \left(\frac {|x+y|}{(2-p)|x+y|+(p-1)|x|}\right)^{\frac 1 {p-2}} x  & \textrm{if  $|x|<|x+y|$}\\
	 x & \textrm{if  $|x+y|\le |x|$}
\end{array} \right..
$$

\noindent (ii) For $p\geq 2$,\footnote{Since for $p=2$ the coefficient $p(p-2)$ vanishes, the exact definition of $w$ is irrelevant in this case.}

$$|x+y|^p\ge |x|^p+p|x|^{p-2}x\cdot y+\frac {1-\kappa} 2\left( p |x|^{p-2} |y|^2+ p(p-2)|w|^{p-2}\bigl( |x|-|x+y| \bigr)^2\right)+c_0|y|^p,$$
where
$$
w =w(x,x+y):= \left\{ \begin{array}{cl}
	x & \textrm{if  $|x|\leq |x+y|$}\\
	\left(\frac {|x+y|}{|x|}\right)^{\frac 1 {p-2}}(x+y) & \textrm{if  $|x+y|\le |x|$}
\end{array} \right..
$$
\end{lem}

\begin{rem}
Note that the constant $c_0$ appearing in the statement above is said to depend on $p$ and $\kappa$, but not on the dimension $n$. The reason is that, to prove the inequality, one can always restrict to the 2-dimensional plane generated by $x$ and $y$, therefore the dimension $n$ of the ambient space plays no role. 
\end{rem}

\begin{rem}
One may be tempted to define directly the weight $\tilde w:=|w|^{p-2}$ with $w$ as above, and then use $\tilde w$ in place of $|w|^{p-2}$ everywhere. However our notation has the advantage that $w\to x$ as $y \to 0$. Not only this emphasizes better the similarities with a Taylor expansion, but it will also be convenient in the proof of Proposition \ref{new spectrum gap}. 
\end{rem}
\begin{proof}
We split the proof in several steps.

\smallskip

\noindent 
$\bullet$ {\it Proof of (i): the case $1<p<2$}.
By approximation we can assume that $|x|\neq 0$.

\smallskip

\noindent 
{\it - Step (i)-1: we show that}
\begin{equation}
\label{base 1}
|x+y|^p\ge \left(1-\frac 1 2 p\right) |x|^p +\frac 1 2 p |x|^{p-2}|x+y|^2 +\frac 1 2 p(p-2)|w|^{p-2}\bigl( |x|-|x+y| \bigr)^2.
\end{equation} 
To prove this, we set  $z=x+y$ and distinguish two cases.

In the case $|z|<|x|$ we set $t:=\frac{|z|}{|x|}$. Then \eqref{base 1} is equivalent to proving that
$$h(t):=t^p - \left(1-\frac 1 2 p\right) -\frac 1 2 p  t^2 -\frac 1 2 p(p-2) (1-t)^2\ge 0,\qquad \forall\,0<t<1.$$
For this, it suffices to notice that $h(1)=h'(1)=0$, and that
$$h''(t)=p\left((p-1)t^{p-2}-1+(2-p)\right)\ge 0\qquad \forall\,0<t<1$$
as $1<p<2$. So \eqref{base 1} holds for $|z|<|x|$. 

On the other hand, in the case $|z|\geq |x|$ we set $t:=\frac{|x|}{|z|}$
and we claim that
$$h(t):=1-\left(1-\frac 1 2 p\right) t^p - \frac 1 2 p t^{p-2}  - \frac 1 2 p(p-2)\frac {1}{(2-p) +(p-1)t}t^{p-2}(t-1)^2\ge 0,\qquad\forall\, 0<t\leq 1.$$
Since $h(1)=0$ and
\begin{align*}
h'(t)=&\frac 1 2 p(p-2)\left[ t^{p-1} - t^{p-3} +2 (1-t) t^{p-2} [(2-p) +(p-1)t]^{-1}\right.\\
&\left. +(2-p)t^{p-3}(t-1)^2[(2-p) +(p-1)t]^{-1}+(p-1)t^{p-2}(t-1)^2[(2-p) +(p-1)t]^{-2}\right]\\
=&\frac 1 2 p(p-2) (t-1)t^{p-3}\left[ t+1 -\frac {2t} {(2-p) +(p-1)t} +\frac {(p-2)(1-t)}{(2-p) +(p-1)t}-\frac{(p-1)t(1-t)}{((2-p) +(p-1)t)^2}\right]\\
=& -\frac 1 2 p(2-p) \frac{ t^{p-2}} {(2-p) +(p-1)t} (p-1)  (t-1)^2 \left[  1+\frac {1}{(2-p) +(p-1)t}\right]\le 0\qquad \forall\,0\leq t\leq 1,
\end{align*} 
we deduce that $h(t)\ge h(1)=0$, concluding the proof of \eqref{base 1}. 

\smallskip
\smallskip

\noindent 
{\it - Step (i)-2: we prove that, for any $x \neq 0$, the function
$$G(x,y):=p |x|^{p-2} |y|^2+ p(p-2)|w|^{p-2}\bigl( |x|-|x+y| \bigr)^2$$
satisfies the lower bound
\begin{equation}
\label{eq:lower G}
G(x,y)\geq c(p)\frac{|x|}{|x|+|y|} |x|^{p-2}|y|^2,\qquad \text{for some $c(p)>0$.}
\end{equation}}
Indeed, when  $|x+y|<|x|$, by the triangle inequality and the fact that $1<p<2$ we get
$$
G(x,y)=p |x|^{p-2}\Bigl(|y|^2-(2-p)\bigl( |x|-|x+y| \bigr)^2\Bigr) \geq p |x|^{p-2}\Bigl(|y|^2-(2-p)|y|^2\Bigr)=p(p-1)|x|^{p-2}|y|^2,
$$
which implies \eqref{eq:lower G}.
On the other hand,  when $|x+y|\geq |x|>0$ we note that
$$|w|^{p-2}=\frac {|x+y|}{(2-p)|x+y|+(p-1)|x|}|x|^{p-2}.$$
Therefore, using again the triangle inequality,
\begin{align*}
G(x,y) &\ge p  \left(|x|^{p-2} |y|^2+  (p-2) |w|^{p-2} |y|^2\right)\\
&=p|x|^{p-2}|y|^2 \frac{(p-1)|x|}{(2-p)|x+y|+(p-1)|x|}\geq p|x|^{p-2}|y|^2 \frac{(p-1)|x|}{(2-p)|y|+|x|},
\end{align*}
and \eqref{eq:lower G} follows.

\smallskip

\noindent 
{\it - Step (i)-3: conclusion.}
As a consequence of  \eqref{eq:lower G},  we know that $G(x,y)\geq 0$ and it vanishes only if $y=0$ (by assumption $x\neq 0$).
Thanks to this fact and recalling \eqref{base 1}, we get the following:
for any $\kappa>0$ and $x \neq 0$, the inequality 
$$|x+y|^p\ge |x|^p+p|x|^{p-2}x\cdot y+\frac {1-\kappa} 2\left( p |x|^{p-2} |y|^2+ p(p-2)|w|^{p-2}\bigl( |x|-|x+y| \bigr)^2\right)$$
holds, and equality is attained  if and only if $y=0$. 

We now prove the inequality in the statement of the lemma by contradiction: If the inequality is false, 
there exist sequences $x_i$ and $y_i$ such that
\begin{multline}
|x_i+y_i|^p\le  |x_i|^p+p|x_i|^{p-2}x_i\cdot y_i+\frac {1-\kappa} 2\left( p |x_i|^{p-2} |y_i|^2+ p(p-2)|w_i|^{p-2}\bigl( |x_i|-|x_i+y_i| \bigr)^2\right) \\
 \qquad + \frac 1 i \min \left\{|y_i|^p,\,|x_i|^{p-2}|y_i|^2\right\},\label{eq:xi yi}
\end{multline}
where $w_i$ corresponds to $x_i$ and $x_i+y_i$. 
By homogeneity (rescaling both $x_i$ and $y_i$ by the same factor $\frac{1}{|x_i|}$) we may assume that $|x_i|=1$, and up to passing to a subsequence we can assume that $x_i\to \bar x$ as $i\to \infty.$

Note that, when $|y_i|$ is large enough, the left hand side in \eqref{eq:xi yi} behaves like $|y_i|^p$
while the right hand side is bounded by $C(p)|y_i|+\frac1i |y_i|^p$.
This implies that the sequence $y_i$ is uniformly bounded,
and up to a subsequence $y_i$ converges to $\bar y$.
Hence, taking the limit in \eqref{eq:xi yi} we deduce that
$$
|\bar x+\bar y|^p\le |\bar x|^p+p|\bar x|^{p-2}\bar x\cdot \bar y+\frac {1-\kappa} 2\left( p |\bar x|^{p-2} |\bar y|^2+ p(p-2)|\bar w|^{p-2}\bigl( |\bar x|-|\bar x+\bar y| \bigr)^2\right),
$$
which is possible only if $\bar y=0$.
This means that  $y_i \to 0$.  However, for $|x|=1$ and $|y|\ll 1$, it follows from a Taylor expansion that
$$
|x+y|^p-\left[ |x|^p+p|x|^{p-2}x\cdot y+\frac {1-\kappa} 2\left( p |x|^{p-2} |y|^2+ p(p-2)|w|^{p-2}\bigl( |x|-|x+y| \bigr)^2\right)\right]\ge \frac \kappa 3 |y|^2,
$$
which is incompatible with \eqref{eq:xi yi} when $i>\frac3{\kappa}$ (since $y_i$ is converging to $0$).
This leads to a contradiction and proves the lemma when $1<p<2$.

\smallskip

\noindent
$\bullet$ {\it Proof of (ii): the case $p\geq 2$}.
By approximation we can assume that $|x+y|\neq 0$ and $|x|\neq 0$.

\smallskip

\noindent 
{\it - Step (ii)-1: we show that}
\begin{equation}
\label{base}
|x+y|^p\ge |x|^p+p|x|^{p-2}x\cdot y+\frac {1} 2\left( p |x|^{p-2} |y|^2+ p(p-2)|w|^{p-2}\bigl( |x|-|x+y| \bigr)^2\right).
\end{equation}
Setting $z=x+y$, this is equivalent to proving that 
$$|z|^p\ge \left(1-\frac 1 2 p\right) |x|^p +\frac 1 2 p |x|^{p-2}|z|^2 +\frac 1 2 p(p-2)|w|^{p-2}\bigl( |x|-|z| \bigr)^2. $$
Set $f(z):=|z|^p$ and 
$$g(z):=\left(1-\frac 1 2 p\right) |x|^p +\frac 1 2 p |x|^{p-2}|z|^2 +\frac 1 2 p(p-2)|w|^{p-2}\bigl( |x|-|z| \bigr)^2. $$

In the case  $|z|\ge |x|$ we note that $f=g$ and $Df=Dg$ on $\partial B(0,\,|x|)$.
Also,
$$D^2f(z)\frac {z}{|z|}\cdot \frac {z}{|z|}= p(p-1)|z|^{p-2}\ge p(p-1)|x|^{p-2}=D^2g(z) \frac {z} {|z|}\cdot \frac {z} {|z|}\qquad \forall\,|z|\geq |x|.$$
Hence, integrating the Hessian of $f-g$ along the segment $\left[\frac{|x|}{|z|}z,z\right]$, we obtain that $f(z)\ge g(z)$ for $|z|\ge|x|$.

On the other hand, in the case $|z|<|x|$, our aim is to prove that
$$|z|^p\ge \left(1-\frac 1 2 p\right) |x|^p +\frac 1 2 p |x|^{p-2}|z|^2 +\frac 1 2 p(p-2) \frac {|z|}{|x|} |z|^{p-2}\bigl( |x|-|z| \bigr)^2.$$
Setting $t:=\frac{|x|}{|z|}$, this is equivalent to saying that 
$$h(t):=1-\left(1-\frac 1 2 p\right) t^p -\frac 1 2 p t^{p-2}  -\frac 1 2 p(p-2) \frac {(t-1)^2}{t} \geq 0\qquad \forall\,t\geq 1.$$
Since $p\geq 2$, a direct computation shows that, for $t\geq 1$,
\begin{align*}
h'(t)&= -\left(1-\frac 1 2 p\right) p t^{p-1}-\frac 1 2 p (p-2)t^{p-3}-  p (p-2) t^{-1}(t-1)+\frac 1 2 p (p-2) t^{-2}(t-1)^2 \\
& =\frac 1 2 p (p-2)\left[ t^{p-1} -t^{p-3}- 2t^{-1}(t-1)+ t^{-2}(t-1)^2\right]\\
& = \frac 1 2 p (p-2)\frac {t-1}{t^2}\left[ t^{p-1} (t+1)- 2t +  (t-1)\right]=\frac 1 2 p (p-2)\frac {t-1}{t^2} (t^{p-1}-1) (t+1) \ge 0.
\end{align*}
Since $h(1)=0$, this implies that $h(t)\ge h(1)=0$ for $t\geq 1$, as desired.
This concludes the proof of \eqref{base}.

\smallskip

\noindent 
{\it - Step (ii)-2: conclusion.}
Thanks to Step (ii)-1 we deduce that, for any $\kappa>0$ and $x\neq 0$, the inequality 
$$|x+y|^p\ge |x|^p+p|x|^{p-2}x\cdot y+\frac {1-\kappa} 2\left( p |x|^{p-2} |y|^2+ p(p-2)|w|^{p-2}\bigl( |x|-|x+y| \bigr)^2\right)$$
becomes an equality if and only if $y=0$ (note that, since $p\geq 2$, the last term above is trivially positive for $y \neq 0$). 
So, if the statement of the lemma does not hold, we can find sequences $x_i$ and $y_i$ such that 
$$|x_i+y_i|^p\le |x_i|^p+p|x_i|^{p-2}x_i\cdot y_i+\frac {1-\kappa} 2\left( p |x_i|^{p-2} |y_i|^2+ p(p-2)|w_i|^{p-2}\bigl( |x_i|-|x_i+y_i| \bigr)^2\right)+ \frac 1 i |y_i|^p,$$
where $w_i$ corresponds to $x_i$ and $x_i+y_i$. 
As before, by homogeneity we may assume that $|x_i|=1$, and that  $x_i\to \bar x$ as $i\to \infty.$
Also, since the left hand side above behaves like $|y_i|^p$ for $|y_i|\gg 1$ while the right hand side is bounded by $(1-\kappa)\frac{p(p-1)}{2}|y_i|^2+\frac1i |y_i|^p$, as $\kappa>0$ and $p\geq 2$ we deduce that $y_i$ cannot go to $\infty$.  This implies that $y_i$ are uniformly bounded,
and as in the previous case we deduce that the only possibility is that $y_i \to 0$.  However,  since 
$$
|x+y|^p-\left[ |x|^p+p|x|^{p-2}x\cdot y+\frac {1-\kappa} 2\left( p |x|^{p-2} |y|^2+ p(p-2)|w|^{p-2}\bigl( |x|-|x+y| \bigr)^2\right)\right]\ge \frac \kappa 3 |y|^2,
$$
for $|x|=1$ and $|y| \ll 1$, this leads to a contradiction when $i>\frac3{\kappa}$.
\end{proof}

We end this section with the following simple lemma.

\begin{lem}\label{upper bound}
\noindent (i) Let $1<p\le  \frac {2n} {n+2}$. For any $\kappa>0$ there exists $C_1=C_1(p^*,\kappa)>0$ such that, for every $a,b\in \mathbb R$, $a\neq 0$, we have
	$$|a+b|^{p^*}\le |a|^{p^*}+p^*|a|^{p^*-2}ab+ \left(\frac{p^*(p^*-1)} 2 +\kappa \right)\frac {(|a|  +C_1|b|)^{p^*}}{|a|^2+|b|^2}|b|^2.$$

	\noindent (ii)	Let $\frac {2n} {n-2}< p <\infty$. For any $\kappa>0$ there exists $C_1=C_1(p^*,\kappa)>0$ such that, for every $a,b\in \mathbb R$, $|a|\neq 0$, 
	$$|a+b|^{p^*}\le |a|^{p^*}+p^*|a|^{p^*-2}ab+\left(\frac{p^*(p^*-1)} 2 +\kappa\right)|a|^{p^*-2}|b|^2+ C_1|b|^{p^*}.$$

\end{lem}
\begin{proof}
Note that (ii) follows from \cite[Lemma 3.2]{FN2019}, so we only need to show (i).
Observe that in this case $p^*\le 2$.

Setting $t:=\frac {b}{a}$, our statement is equivalent to proving that
\begin{equation}
\label{target}
|1+t|^{p^*}-1-p^*t- \left(\frac{p^*(p^*-1)} 2 +\kappa \right) \frac {(1+C_1|t|)^{p^*}}{1+|t|^2}|t|^2 \le 0
\end{equation} 
for any $t\in \mathbb R$ and  some $C_1>0$. 

First of all, by a Taylor expansion,
$$|1+t|^{p^*}=1+p^*t+\frac {p^*(p^*-1)} 2 |t|^2 +o(|t|^2)\qquad \forall\,|t|\ll1.$$
Also, by  the concavity of $t\mapsto t^{\frac 1 {p^*}}$  we have
$$ 1+\frac 1 {p^*} |t|^2 \ge (1+|t|^2)^{\frac 1 {p^*}}\qquad \forall\,|t|\leq 1.$$
Therefore there exists $t_0=t_0(n,p)>0$ small such that, for any $C_1\ge \frac 1 {p^*}$,
$$|1+t|^{p^*}-1-p^*t\le  \left(\frac{p^*(p^*-1)} 2 +\kappa \right) \frac {(1+C_1|t|)^{p^*}}{1+|t|^2}|t|^2\qquad \forall \,t \in [-t_0,t_0].$$
On the other hand, for $|t|>t_0$ we can rewrite \eqref{target} as
\begin{equation}\label{target2}
\Biggl[\Biggl((1+|t|^2)\frac{|1+t|^{p^*}-1-p^*t}{ \left(\frac{p^*(p^*-1)} 2 +\kappa \right)|t|^2}\Biggr)^{\frac 1 {p^*}}-1\Biggr] |t|^{-1} \le C_1.
\end{equation}
Since
the left-hand side of \eqref{target2} is bounded as $|t|\to +\infty$, the existence of a constant $C_1<+\infty$ such that \eqref{target2} holds on $\R\setminus (-t_0,t_0)$ follows by compactness.
\end{proof}

\section{Spectral gaps}\label{spectrum gaps}

Let $v=v_{a,b,x_0} \in \mathcal M$. The goal of this section is to study some embedding/compacteness theorems and spectral gaps for weighted Sobolev/Orlicz-type spaces, where the weights depend on $v$.
Throughout this section we assume that $a_0>0$, $b=1$, and $x_0=0$, that is
$$
v(x)=\frac{a_0}{\left(1+|x|^{\frac{p}{p-1}}\right)^{\frac{n-p}{p}}},
$$
where $a_0>0$ is any constant such that $\frac12 \leq \|v\|_{L^{p^*}}\leq 2.$

Given $\Omega\subset \R^n$, $q\geq 1,$ and a non-negative locally integrable function $g_0:\R^n\to \R$, we define the Banach space $L^q(\Omega;g_0)$ as the space of measurable functions $\varphi:\Omega\to \R$ whose norm
$$\|\varphi \|_{L^q(\Omega;\,g_0)}:=\left(\int_{\Omega} |\varphi|^q \,g_0(x)\, dx\right)^{\frac 1 q}$$
is finite. Also, given  $g_1 \in L^1_{\rm loc}(\R^n\setminus \{0\})$ non-negative, 
we denote by $C^1_{c,0}(\R^n)$ the space of compactly supported functions of class $C^1$ that are constant in a neighborhood of the origin, and we
define $\dot{W}^{1,q}(\mathbb R^n;g_1)$ as the closure of $C^1_{c,0}(\R^n)$ with respect to the norm
$$\|\varphi\|_{\dot{W}^{1,q}(\mathbb R^n;\,g_1)}:=\left(\int_{\mathbb R^n}|D\varphi|^q \,g_1(x) \,dx\right)^{\frac 1 q}.$$

\begin{rem}
\label{rem:weight}
It is important for us to consider weights that are not necessarily integrable at the origin, since $|Dv|^{p-2}\sim |x|^{\frac{p-2}{p-1}} \not\in L^1(B_1)$ for $p \leq \frac{n+2}{n+1}$.
This is why, when defining weighted Sobolev spaces, we consider the space $C^1_{c,0}(\R^n)$, so that gradients vanish near $0$.
Of course, replacing $C^1_{c}(\R^n)$ by $C^1_{c,0}(\R^n)$ plays no role in the case $p>\frac{n+2}{n+1}$.
\end{rem}
\subsection{Compact embedding}

The following embedding theorem generalizes \cite[Corollary 6.2]{FN2019}. 

\begin{prop}\label{compact embedding}
	Let $1<p<\infty$. The space	$\dot{W}^{1,2}(\mathbb R^n;|Dv|^{p-2})$ compactly embeds into $L^2(\mathbb R^n;v^{p^*-2})$.
\end{prop}

To prove this result, we first show an intermediate estimate that will be useful also later.
\begin{lem}\label{compact embedding 2}
	Let $1<p<\infty$, and $\varphi\in \dot{W}^{1,2}(\mathbb R^n;|Dv|^{p-2})\cap L^2(\mathbb R^n;v^{p^*-2})$.
	Then
	\begin{equation}
	\label{embed}
	\int_{\mathbb R^n} v^{p^*-2} |\varphi|^2 \, dx\le C(n,p) \int_{\mathbb R^n}|Dv|^{p-2} |D\varphi|^2 \, dx.
	\end{equation}
	Also, there exists $\vartheta=\vartheta(n,p)>0$ such that, for any $\rho \in (0,1)$, we have
	\begin{equation}
	\label{embed: small}
\int_{B(0,\rho)} v^{p^*-2} |\varphi|^2\, dx\leq C(n,p) \rho^{\vartheta} \int_{\mathbb R^n}|Dv|^{p-2} |D\varphi|^2 \, dx
\end{equation}
and
\begin{equation}
	\label{embed: large}
\int_{\R^n\setminus B(0,\rho^{-1})} v^{p^*-2} |\varphi|^2\, dx \le \frac{C(n,p)}{|\log\rho|^{2}} \int_{\mathbb R^n}|Dv|^{p-2} |D\varphi|^2 \, dx.
\end{equation}	 

\end{lem}

\begin{proof}
To prove \eqref{embed}, we can assume by approximation that $\varphi \in C^1_{c,0}(\R^n)$
(see Remark \ref{rem:weight}).

We note that, thanks to Fubini's theorem and using polar coordinates, 
	\begin{align*}
	\int_{\mathbb R^n} v^{p^*-2} |\varphi|^2 \, dx
	& \leq  C(n,p) \int_{\mathbb S^{n-1}}\int_{0}^{\infty} r^{n-1} (1+ r^{\frac p {p-1}})^{-\frac {n(p-2)} {p}-2 } |  \varphi(r\theta)|^2 \,dr\,d\theta \\
	& \le  C(n,p) \int_{\mathbb S^{n-1}}\int_{0}^{\infty} r^{n-1} (1+  r^{\frac p {p-1}})^{-\frac {n(p-2)} {p}-2 } \int_{r}^{\infty}   |\varphi(t\theta)||D\varphi(t\theta)|\,dt \,dr\, d\theta \\
	& \le  C(n,p) \int_{\mathbb S^{n-1}}\int_{0}^{\infty} \int_{0}^{t}   |\varphi(t\theta)||D\varphi(t\theta)|  r^{n-1} (1+  r^{\frac p {p-1}})^{-\frac {n(p-2)} {p}-2 } \,dr \,dt\, d\theta\\
	& \le   C(n,p) \int_{\mathbb S^{n-1}}\int_{0}^{\infty}  |\varphi(t\theta)||D\varphi(t\theta)|   t^{n} (1+  t^{\frac p {p-1}})^{-\frac {n(p-2)} {p}-2 } \,dt\, d\theta. 
	\end{align*}
	Thus, by Cauchy-Schwarz inequality we get
	\begin{align*}
	 \int_{\mathbb R^n} v^{p^*-2} |\varphi|^2 \, dx &
 \le  C(n,p) \biggl(\int_{\mathbb S^{n-1}}\int_{0}^{\infty}   |D\varphi(t\theta)|^2   t^{n+1} (1+  t^{\frac p {p-1}})^{-\frac {n(p-2)} {p}-2 } \,dt\, d\theta\biggr)^{1/2}\cdot\\
&\qquad\qquad\qquad\qquad\cdot \biggl(\int_{\mathbb S^{n-1}}\int_{0}^{\infty} t^{n-1} (1+ t^{\frac p {p-1}})^{-\frac {n(p-2)} {p}-2 } |  \varphi(t\theta)|^2 \,dt\,d\theta\biggr)^{1/2},
	\end{align*}
	and since the last term in the right hand side coincides with $\|\varphi \|_{L^2(\R^n;v^{p^*-2})}$ (up to a multiplicative constant), we conclude that
	\begin{equation}\label{intermediate}
	\begin{split}
 \int_{\mathbb R^n} v^{p^*-2} |\varphi|^2 \, dx&
 \le  C(n,p) \int_{\mathbb S^{n-1}}\int_{0}^{\infty}   |D\varphi(t\theta)|^2   t^{n+1} (1+  t^{\frac p {p-1}})^{-\frac {n(p-2)} {p}-2 } \,dt\, d\theta\\
 & \le  C(n,p) \int_{\R^n}   |D\varphi(x)|^2   |x|^{2} (1+  |x|^{\frac p {p-1}})^{-\frac {n(p-2)} {p}-2 } \,dx.
 \end{split}
	\end{equation}	
	We now observe that
	$$
	|x|^2 \sim |x|^{1+\frac 1 {p-1}} |Dv|^{p-2} \leq C(n,p)\, |Dv|^{p-2}\qquad  \text{ when } |x|\in (0,1],
      $$
	and 
	\begin{equation*}
	\begin{split} |x|^{2} (1+  |x|^{\frac p {p-1}})^{-\frac {n(p-2)} {p}-2 }
	 \sim   |x|^{-\frac p {p-1}} |Dv|^{p-2}\leq C(n,p)\, |Dv|^{p-2} \qquad\text{ when } |x|\in (1,\,\infty),
       \end{split}\end{equation*}
so \eqref{embed} follows from \eqref{intermediate}.

\smallskip

To prove \eqref{embed: small}, we apply \eqref{embed} and the Sobolev inequality with radial weights (see e.g.\ \cite[Section 2.1]{M1995}). More precisely, since $|Dv|^{p-2}\geq c(n,p) |x|$ inside $B(0,1)$, \begin{multline*}
	 \int_{\mathbb R^n}|Dv|^{p-2} |D\varphi|^2 \, dx \geq c(n,p)\int_{\mathbb R^n} \bigl( v^{p^*-2} |\varphi|^2+|Dv|^{p-2} |D\varphi|^2\bigr) \, dx\\
	 \geq c(n,p)\int_{B(0,1)} \bigl(  |\varphi|^2+|x|\,|D\varphi|^2\bigr) \, dx
	\geq \biggl(\int_{B(0,1)}|\varphi|^{q}\,dx\biggr)^{\frac{2}{q}},
	\end{multline*}
	where $q=q(n)>2$. Thus, by H\"older inequality, for any $\rho\in (0,1)$ we get
	\begin{multline*}
	\int_{B(0,\rho)} v^{p^*-2} |\varphi|^2\, dx\leq C(n,p)\int_{B(0,\rho)} |\varphi|^2\, dx\\
	\leq C(n,p)\rho^{n\left(1-\frac{2}{q}\right)}\biggl(\int_{B(0,\rho)}|\varphi|^{q}\,dx\biggr)^{\frac{2}{q}}
	\leq C(n,p)\rho^{n\left(1-\frac{2}{q}\right)}  \int_{\mathbb R^n}|Dv|^{p-2} |D\varphi|^2 \, dx,
	\end{multline*}
	as desired.

\smallskip

To prove \eqref{embed: large}, we define
 $$
 \chi_\rho(x):=\left\{
 \begin{array}{ll}
 0 &\text{for }|x|<\rho^{-1/2}\\
 \frac{2\log|x|- |\log \rho|}{|\log \rho|}&\text{for }\rho^{-1/2}\leq |x|\leq \rho^{-1}\\
 1 & \text{for }\rho^{-1}\leq |x|
 \end{array}
 \right.
 $$
and we apply  \eqref{intermediate} to the function $\phi_\rho:=\chi_\rho \varphi$:
\begin{align*}
\int_{\R^n\setminus B(0,\rho^{-1})} v^{p^*-2} |\varphi|^2\, dx
&\le \int_{\R^n} v^{p^*-2} |\phi_\rho|^2\, dx\le C(n,p) \int_{\mathbb R^n} |x|^{2} (1+  |x|^{\frac p {p-1}})^{-\frac {n(p-2)} {p}-2 }  |D\phi_\rho|^2 \, dx\\
& \le C(n,p) \int_{\mathbb R^n} |x|^{2} (1+  |x|^{\frac p {p-1}})^{-\frac {n(p-2)} {p}-2 }\chi_\rho^2\, |D\varphi|^2 \, dx\\
&\qquad+ C(n,p) \int_{\R^n} |x|^{2} (1+  |x|^{\frac p {p-1}})^{-\frac {n(p-2)} {p}-2 } |D\chi_\rho|^2\varphi^2 \, dx\\
&\le C(n,p) \int_{\mathbb R^n\setminus B(0,\rho^{-1/2})}|x|^{2} (1+  |x|^{\frac p {p-1}})^{-\frac {n(p-2)} {p}-2 } |D\varphi|^2 \, dx\\
&\qquad+ C(n,p) |\log \rho|^{-2} \int_{B(0,\rho^{-1})\setminus B(0,\rho^{-1/2})} (1+  |x|^{\frac p {p-1}})^{-\frac {n(p-2)} {p}-2 } \varphi^2 \, dx\\
&\le C(n,p)\rho^{\frac{p}{p-1}} \int_{\mathbb R^n\setminus B(0,\rho^{-1})}|Dv|^{p-2} |D\varphi|^2 \, dx\\
&\qquad+ C(n,p) |\log \rho|^{-2}  \int_{B(0,\rho^{-1})\setminus B(0,\rho^{-1/2})}|v|^{p^*-2}  \varphi^2 \, dx\\
&\le C(n,p) |\log \rho|^{-2} \int_{\mathbb R^n}|Dv|^{p-2} |D\varphi|^2 \, dx,
\end{align*}
where the last inequality follows from \eqref{embed}.
\end{proof}

\begin{proof}[Proof of Proposition~\ref{compact embedding}]
Let $\varphi_i$ be a sequence of functions in $\dot{W}^{1,2}(\mathbb R^n;|Dv|^{p-2})$ 
with uniformly bounded norm. 
It follows by \eqref{embed} that their $L^2(\mathbb R^n;v^{p^*-2})$ norm is uniformly bounded as well.

Since both $|Dv|^{p-2}$ and $v^{p^*-2}$ are locally bounded away from zero and infinity in $\R^n\setminus \{0\}$, by Rellich-Kondrachov Theorem and a diagonal argument we deduce that, up to a subsequence, $\varphi_i$ converges to some function 
$\varphi$ both weakly in $\dot{W}^{1,2}(\R^n;|Dv|^{p-2})\cap L^2(\R^n;v^{p^*-2})$ and strongly in $L^2_{\rm loc}(\R^n\setminus \{0\};v^{p^*-2})$.

Also, it follows by \eqref{embed: small} and \eqref{embed: large} that, for any $ \rho\in (0,1)$,
$$
\int_{\R^n\setminus B(0,\rho)} v^{p^*-2} |\varphi_i|^2\, dx \le C(n,p) \rho^{\vartheta},\qquad
 \int_{\mathbb R^n\setminus B(0,\rho^{-1})} v^{p^*-2} |\varphi_i|^2\, dx\le  \frac{C(n,p)}{|\log \rho|^2} . 
$$
We conclude the proof  by  defining the compact set $K_\rho:=\overline{B(0,\rho^{-1})}\setminus B(0,\rho)$ and applying the strong convergence of $\varphi_i$ on $K_\rho$, together with the arbitrariness of $\rho$ (that can be chosen arbitrarily small).
\end{proof}

As we shall see, the previous result allows us to deal with the case $p> \frac {2n}{n+2}$.
However, when $1<p\le \frac {2n}{n+2}$, we will need a much more delicate compactness result that we now present.

\begin{lem}\label{weighted compact}
	Let $1<p\le \frac {2n}{n+2}$, and let $\phi_i$ be a sequence of functions in  $\dot W^{1,p}(\mathbb R^n)$ satisfying
\begin{equation}
\label{eq:weighted W12}
\int_{\mathbb R^n} \bigl( |Dv|+\ez_i|D\phi_i| \bigr)^{p-2} |D\phi_i|^2\, dx\le  1,
\end{equation}
	where $\ez_i \in (0,1)$ is a sequence of positive numbers converging to $0$. Then, up to a subsequence, $\phi_i$ converges weakly in $\dot W^{1,p}(\R^n)$ to some function $\phi\in \dot W^{1,p}(\mathbb R^n)\cap L^2(\R^n;v^{p^*-2})$.
	Also, given any constant $C_1\geq 0$ it holds\footnote{As already noticed in the introduction, the expression appearing in the left hand side of \eqref{eq:L2 conv p small} behaves 
like $v^{p^*-2}|\phi_i|^2$ when $|\phi_i|\ll \frac{v}{\ez_i}$, and like $\ez_i^{p^*-2}|\phi_i|^{p^*}$ otherwise.
Analogously, the expression in \eqref{eq:weighted W12} behaves like $|Dv|^{p-2}|D\phi_i|^2$ when $|D\phi_i|\ll \frac{|Dv|}{\ez_i}$, and like $\ez_i^{p-2}|D\phi_i|^{p}$ otherwise.
These substantial changes of behavior, and the fact that the change in size of the gradients does not necessarily correspond to a change in size of the functions, make the proof particularly delicate.}
	\begin{equation}
	\label{eq:L2 conv p small} \int_{\mathbb R^n} \frac{(v+C_1\ez_i\phi_i)^{p^*}}{v^2+|\ez_i\phi_i|^2} |\phi_i|^2\, dx \to \int_{\mathbb R^n} v^{p^*-2}|\phi|^2\, dx\qquad \text{as }i\to \infty.
	\end{equation}
\end{lem}

%\begin{rem}
%To simplify the notation, 
%throughout the proof we will consider $C_1$ as a constant depending on $n$ and $p$, so that whenever  a constant $C$ depends on $C_1$ we shall write $C=C(n,p)$. The reason for this is that, later on, we will apply this result with $C_1$ the constant in Lemma~\ref{upper bound} for some small value $\kappa=\kappa(n,p)>0$.
%\end{rem}

\begin{proof} 
Up to replacing $\phi_i$ by $|\phi_i|,$ we can assume that 
	$\phi_i\ge 0$. Note that $p<p^*\le 2$ under our assumption.
	
	Observe that, by H\"older inequality,
		\begin{align*}
	\int_{\mathbb R^n}  |D\phi_i|^p \, dx
	&\le \biggl(\int_{\mathbb R^n} \bigl( |Dv|+\ez_i|D\phi_i| \bigr)^{p-2} |D\phi_i|^2 \, dx\biggr)^{\frac p2}\biggl(\int_{\mathbb R^n} \bigl( |Dv|+\ez_i|D\phi_i| \bigr)^{p} \, dx\biggr)^{1-\frac p2}\\
	&\le C(n,p)\biggl(\int_{\mathbb R^n} \bigl( |Dv|+\ez_i|D\phi_i| \bigr)^{p-2} |D\phi_i|^2 \, dx\biggr)^{\frac p2}\biggl(1+\ez_i^p\int_{\mathbb R^n}  |D\phi_i|^p \, dx\biggr)^{1-\frac p2},
	\end{align*}
	that combined with \eqref{eq:weighted W12} gives
	\begin{equation}
	\label{Lp phi}
	\left(\int_{\mathbb R^n}  |D\phi_i|^p \, dx\right)^{\frac 2 p}\le C(n,p) \int_{\mathbb R^n} \bigl( |Dv|+\ez_i|D\phi_i| \bigr)^{p-2} |D\phi_i|^2 \, dx \leq C(n,p).
	\end{equation}
	Thus, up to a subsequence, $\phi_i$ converges weakly in $\dot W^{1,p}(\mathbb R^n)$ and also a.e. to some function  $\phi \in \dot W^{1,p}(\mathbb R^n)$. 
	Hence, to conclude the proof, we need to show the validity of \eqref{eq:L2 conv p small}.

	We first prove it under the assumption that
	$\ez_i \phi_i\le \zeta v$
	with some small constant $\zeta=\zeta(n,p,C_1)\in (0,1)$ be determined.
	Later, we will remove this assumption.
	
	\smallskip

\noindent 
$\bullet$ {\it Step 1: proof of \eqref{eq:L2 conv p small} when $\ez_i \phi_i\le \zeta v$}.
	Since $\ez_i \phi_i $ is bounded by $\zeta v\leq v$, we  have that $\biggl(1+\frac{\ez_i \phi_i}{v}\biggr)\leq 2$, thus
	\begin{align*}
	\int_{\mathbb R^n} (v+\ez_i \phi_i)^{p^*-2}|\phi_i|^2\, dx
	&\le  \int_{\mathbb R^n} v^{p^*-2} \biggl(1+\frac{\ez_i \phi_i}{v}\biggr)^{p^*-2}|\phi_i|^2\,dx \\
	&\le 2^{p^*-p}  \int_{\mathbb R^n} v^{p^*-2} \biggl(1+\frac{\ez_i \phi_i}{v}\biggr)^{p-2}|\phi_i|^2\,dx.
	\end{align*}
	Recall that 
	\begin{equation}
	\label{v and Dv}
	v\sim (1+|x|^{\frac p {p-1}})^{1-\frac n p} \quad \text{ and }  \quad |Dv|\sim (1+|x|^{\frac p {p-1}})^{-\frac n p} |x|^{\frac 1 {p-1}}, 
	\end{equation} 
	where the constants depend  only on $p$ and $n$. Moreover, the following Hardy-Poincar\'e inequality holds~\cite{S2014}\footnote{More precisely, the case $\gamma>1$ is stated in \cite[Theorem 3.1]{S2014},
	while the case $\gamma=1$ follows from the classical Hardy inequality (see for instace \cite[Theorem 4.1]{S2014}).}: For any $p>1$ and $\gamma\geq 1$, and any compactly supported function $\xi\in W^{1,\, p}(\mathbb R^n)$, one has
	$$\int_{\mathbb R^n} |\xi|^p\left[\left(1+|x|^{\frac p {p-1}}\right)^{p-1}\right]^{\gamma-1}\, dx \le C(n,p,\gamma) \int_{\mathbb R^n} |D\xi|^p\left[\left(1+|x|^{\frac p {p-1}}\right)^{p-1}\right]^{\gamma}\, dx.$$
	By approximation, we can apply this inequality with
	$$
	\gamma=1+\frac{(2-p^*)\left(\frac{n}{p}-1\right)}{p-1}\qquad \text{ and }\qquad \xi=\biggl(1+\frac{\ez_i \phi_i}{v}\biggr)^{\frac {p-2} p}|\phi_i|^{\frac 2 p}.
	$$ Thus, since $v^{p^*-2}\sim \left[\left(1+|x|^{\frac p {p-1}}\right)^{p-1}\right]^{\gamma-1}$, we get
	\begin{equation}
	\begin{split}\label{phi11}
	&\int_{\mathbb R^n} (v+\ez_i \phi_i)^{p^*-2}|\phi_i|^2\, dx\le C(n,p) \int_{\mathbb R^n} v^{p^*-2} \biggl(1+\frac{\ez_i \phi_i}{v}\biggr)^{p-2}|\phi_i|^2\,dx\\
	&\qquad \leq C(n,p) \biggl\|\biggl(1+\frac{\ez_i \phi_i}{v}\biggr)^{\frac {p-2} p} |\phi_i|^{\frac 2 p}\biggr\|^p_{\dot W^{1,p}\bigl(\R^n;v^{p^*-2}\bigl(1+|x|^{\frac p {p-1}}\bigr)^{p-1}\bigr)}\\
	&\qquad\le C(n,p) \int_{\mathbb R^n} v^{p^*-2}\left(1+|x|^{\frac p {p-1}}\right)^{p-1}\cdot \\
	&\qquad \qquad \cdot
	\biggl[
	\biggl(1+\frac{\ez_i \phi_i}{v}\biggr)^{-2}|\phi_i|^2 \left( \frac {\ez_i \phi_i|Dv|}{v^2} + \frac{\ez_i|D\phi_i|}{v}\right)^p + \biggl(1+\frac{\ez_i \phi_i}{v}\biggr)^{p-2}|\phi_i|^{2-p}|D\phi_i|^p
	\biggr]\,dx  \\
	&\qquad\le C(n,p) \int_{\mathbb R^n} v^{p^*-2}\left(1+|x|^{\frac p {p-1}}\right)^{p-1} \left[|\phi_i|^2 \left( \frac { \zeta |Dv|}{v} + \frac{\ez_i|D\phi_i|}{v}\right)^p +|\phi_i|^{2-p}|D\phi_i|^p
	\right]\,dx,
	\end{split}
	\end{equation}
	where, in the last inequality, we used that  $0\leq \frac {\ez_i\phi_i}{v}\le \zeta<1$.
	
	We now apply \eqref{young} to the last integrand in \eqref{phi11} with $\ez=\ez_i, r=|x|, a=|\phi_i|, b=|D\phi_i|$. In this way, thanks to  \eqref{phi11} and since $v+\ez_i\phi_i\leq 2v$, we deduce that for any $\ez_0>0$ there exists $\zeta=\zeta(\ez_0) \in (0,1)$  such that
	\begin{align*}
	\int_{\mathbb R^n} v^{p^*-2}|\phi_i|^2\, dx&\leq 2^{2-p^*}\int_{\mathbb R^n} (v+\ez_i \phi_i)^{p^*-2}|\phi_i|^2\, dx \\
	& \leq C(n,p) \biggl\|\biggl(1+\frac{\ez_i \phi_i}{v}\biggr)^{\frac {p-2} p} |\phi_i|^{\frac 2 p}\biggr\|^p_{\dot W^{1,p}\bigl(\R^n;v^{p^*-2}\bigl(1+|x|^{\frac p {p-1}}\bigr)^{p-1}\bigr)}\\
	&\ \le C(n,p) \ez_0 \int_{\mathbb R^n} v^{p^*-2}|\phi_i|^2\, dx+  C(\ez_0,n,p) \int_{\mathbb R^n} \bigl( |Dv|+\ez_i|D\phi_i| \bigr)^{p-2}|D\phi_i|^2\, dx.	\end{align*}
	Thus, fixing $\ez_0$ small enough so that $C(n,p) \ez_0\le \frac 1 2$, it follows from \eqref{eq:weighted W12} and the inequality above that	
	\begin{multline}
	\label{poincare}
	\int_{\mathbb R^n} v^{p^*-2}|\phi_i|^2\, dx+\biggl\|\biggl(1+\frac{\ez_i \phi_i}{v}\biggr)^{\frac {p-2} p} |\phi_i|^{\frac 2 p}\biggr\|^p_{\dot W^{1,p}\bigl(\R^n;v^{p^*-2}\bigl(1+|x|^{\frac p {p-1}}\bigr)^{p-1}\bigr)}\\
	\le C(n,p) \int_{\mathbb R^n} \bigl( |Dv|+\ez_i|D\phi_i| \bigr)^{p-2}|D\phi_i|^2\, dx\le C(n,p). 
	\end{multline}
	In particular, the sequence $\left(1+\frac{\ez_i \phi_i}{v}\right)^{\frac {p-2} p} |\phi_i|^{\frac 2 p}$ is uniformly bounded in $\dot W^{1,p}_{\rm loc}(\R^n)\subset L^{p^*}_{\rm loc}(\R^n)$.
	Since $\left(1+\frac{\ez_i \phi_i}{v}\right)\sim 1$, this implies that $|\phi_i|^{\frac 2 p}\in L^{p^*}_{\rm loc}(\R^n)$.
	Combining this higher integrability estimate with the a.e. convergence of $\phi_i$ to $\phi$, 
	by dominated convergence we deduce that, for any $R>1,$
		\begin{equation}
	\label{a.e.}
	\int_{B(0,R)} \frac{(v+C_1\ez_i\phi_i)^{p^*}}{v^2+|\ez_i\phi_i|^2} |\phi_i|^2\, dx \to \int_{B(0,R)} v^{p^*-2}|\phi|^2\, dx \qquad \text{as $i \to \infty$}
	\end{equation}	(recall that $\ez_i\to 0$).

	Also, since $1<p\leq \frac{2n}{n+2}$ it follows that $n\geq 3$, and therefore
$$\frac{-np+2n-2p}{p-1}+n=\frac{n-2p}{p-1}> 0.$$ 
	This allows us to apply Lemma~\ref{HP ineq} to $\phi_i$ with
	$$\az=\frac{np-2n+2p}{p-1},$$
	 and similarly to \eqref{phi11} we obtain (recall \eqref{v and Dv})
	\begin{align*}
	& \int_{\mathbb R^n\setminus B(0,R)} \frac{(v+C_1\ez_i\phi_i)^{p^*}}{v^2+|\ez_i\phi_i|^2} |\phi_i|^2\, dx\\
	&\qquad \le  C(n,p,C_1) \int_{\mathbb R^n\setminus B(0,R)} v^{p^*-2} \biggl(1+\frac{\ez_i \phi_i}{v}\biggr)^{p-2}|\phi_i|^2\,dx\\
	&\qquad \le  C(n,p,C_1) \int_{\mathbb R^n\setminus B(0,R)}  |x|^{\frac{-np+2n-2p}{p-1}+p}\,\cdot \\
	&\qquad \qquad \cdot
	\biggl[
	\biggl(1+\frac{\ez_i \phi_i}{v}\biggr)^{-2}|\phi_i|^2 \left( \frac {\ez_i \phi_i|Dv|}{v^2} + \frac{\ez_i|D\phi_i|}{v}\right)^p + \biggl(1+\frac{\ez_i \phi_i}{v}\biggr)^{p-2}|\phi_i|^{2-p}|D\phi_i|^p
	\biggr]\,dx\\
	&\qquad \le C(n,p,C_1) \int_{\mathbb R^n\setminus B(0,R)}   |x|^{\frac{-np+2n-2p}{p-1}+p}\left[|\phi_i|^2 \left( \frac { \zeta |Dv|}{v} + \frac{\ez_i|D\phi_i|}{v}\right)^p +|\phi_i|^{2-p}|D\phi_i|^p
	\right]\,dx.
	\end{align*}
Then, applying \eqref{inter} to the last term above with $\ez=\ez_i, r=|x|, a=|\phi_i|, b=|D\phi_i|$, we obtain that for any $\ez_0'>0$ there exists $\zeta=\zeta(\ez_0') \in (0,1)$  such that
	\begin{align*}
	\int_{\mathbb R^n\setminus B(0,R)}\frac{(v+C_1\ez_i\phi_i)^{p^*}}{v^2+|\ez_i\phi_i|^2} |\phi_i|^2\, dx   
	 &\le  C(n,p,C_1) \ez_0' \int_{\mathbb R^n\setminus B(0,R)} v^{p^*-2}|\phi_i|^2\, dx\\
	 &\qquad+  C(\ez_0',n,p,C_1) R^{-\frac p{p-1}} \int_{\mathbb R^n\setminus B(0,R)}\bigl( |Dv|+\ez_i|D\phi_i| \bigr)^{p-2}|D\phi_i|^2\, dx \\
	 &\le  C(n,p,C_1) \ez_0' \int_{\mathbb R^n\setminus B(0,R)}\frac{(v+C_1\ez_i\phi_i)^{p^*}}{v^2+|\ez_i\phi_i|^2} |\phi_i|^2\, dx\\
	& \qquad +  C(\ez_0',n,p,C_1)  R^{-\frac p{p-1}} \int_{\mathbb R^n\setminus B(0,R)} \bigl( |Dv|+\ez_i|D\phi_i| \bigr)^{p-2}|D\phi_i|^2\, dx. 
	\end{align*}
	Thus, by fixing $\ez_0'$ so that $C(n,p,C_1) \ez_0'\le \frac 1 2$, it follows that
	$$\int_{\mathbb R^n\setminus B(0,R)} \frac{(v+C_1\ez_i\phi_i)^{p^*}}{v^2+|\ez_i\phi_i|^2} |\phi_i|^2\, dx \le C(n,p,C_1) R^{-\frac p{p-1}} \int_{\mathbb R^n}  \bigl( |Dv|+\ez_i|D\phi_i| \bigr)^{p-2}|D\phi_i|^2\, dx\le C(n,p) R^{-\frac p{p-1}}.$$
	Combining this bound with \eqref{poincare} and \eqref{a.e.}, by the arbitrariness of $R$ we conclude that $\phi\in L^2(\R^n;v^{p^*-2})$ and that 
	\eqref{eq:L2 conv p small} holds. This concludes the proof under the assumption that $\ez_i\phi_i\le \zeta v$ with $\zeta=\zeta(n,p,C_1)>0$ sufficiently small.

		\smallskip

\noindent 
$\bullet$ {\it Step 2: proof of \eqref{eq:L2 conv p small} in the general case}.
Throughout this part, we assume that $\zeta=\zeta(n,p,C_1)>0$ is a small constant so that Step 1 applies.

	Observe that, by \eqref{S equ}, ${\zeta v}$ is a supersolution for the operator
	$${L}_v[\psi]:=-{\rm div}\left(\bigl( |Dv|+ |D\psi| \bigr)^{p-2}D\psi+(p-2)\bigl( |Dv|+ |D\psi| \bigr)^{p-3}|D\psi| D\psi \right),$$
	namely ${L}_v[\zeta v]\geq 0$.
	Therefore, multiplying ${L}_v[\zeta v]\geq 0$ by $\left(\ez_i \phi_i-{\zeta v}\right)_+$ and integrating by parts, we get
	\begin{multline}
	\label{eq:sub}
	\int_{\R^n} \bigl( |Dv|+\zeta  |Dv|\bigr)^{p-2} \zeta Dv \cdot D(\ez_i \phi_i-{\zeta v})_+ \,dx \\
	+(p-2)\int_{\R^n} \bigl( |Dv|+\zeta  |Dv|\bigr)^{p-3}\zeta^2 |Dv| Dv \cdot D(\ez_i \phi_i-{\zeta v})_+  \,dx\geq 0.
	\end{multline}
	Also, by the convexity of
	$$\R^n\ni z\mapsto F_t(z):= (t+|z|)^{p-2}|z|^2, \qquad t\ge 0,$$ 
	we have 
	$$
	F_t(z)+DF_t(z)\cdot(z'-z)\leq F_t(z')\qquad \forall\,z,z'\in \R^n,\,t\ge 0.
      $$
      Hence, applying this inequality with $t=|Dv|,$ $z=\zeta Dv$, and $z'=\ez_i D\phi_i$,
      it follows by \eqref{eq:sub} that
	\begin{equation}\begin{split}\label{phi1 large}
	&c(n,p)\ez_i^{-2} \int_{\{\ez_i \phi_i>\zeta v\}}  |Dv|^p\, dx\le \ez_i^{-2} \int_{\{\ez_i \phi_i>\zeta v\}}  \bigl( |Dv|+\zeta  |Dv|\bigr)^{p-2} \zeta^2 |Dv|^2\, dx \\
	&\qquad\le  \ez_i^{-2}  \int_{\{\ez_i \phi_i>\zeta v\}}  \bigl( |Dv|+\zeta  |Dv|\bigr)^{p-2} \zeta^2 |Dv|^2\, dx   \\
	& \qquad\qquad+ \ez_i^{-2}\int_{\{\ez_i \phi_i>\zeta v\}} \bigl( |Dv|+\zeta  |Dv|\bigr)^{p-2} \zeta Dv \cdot D(\ez_i \phi_i-{\zeta v})_+ \,dx\\
	& \qquad\qquad+\ez_i^{-2}(p-2)\int_{\{\ez_i \phi_i>\zeta v\}} \bigl( |Dv|+\zeta  |Dv|\bigr)^{p-3}\zeta^2 |Dv| Dv \cdot D(\ez_i \phi_i-{\zeta v})_+  \,dx \\
	&\qquad\le  \int_{\{\ez_i \phi_i>\zeta v\}}  \bigl(  |Dv|+ \ez_i|D\phi_i|\bigr)^{p-2}  |D\phi_i|^2\, dx.
	\end{split}
	\end{equation}
We now write $\phi_i=\phi_{i,1}+\phi_{i,2}$, where 
	$$\phi_{i,1}:=\min\left\{\phi_i,\,\frac {\zeta v} {\ez_i}\right\},\qquad \phi_{i,2}:=\phi_i-\phi_{i,1}.$$	Note that, as a consequence of \eqref{eq:weighted W12} and \eqref{phi1 large},
	\begin{multline}
	\label{eq:bound phi12}
	\int_{\mathbb R^n}  \bigl(  |Dv|+ \ez_i|D\phi_{i,1}|\bigr)^{p-2}  |D\phi_{i,1}|^2\, dx+
	\int_{\mathbb R^n}  \bigl(  |Dv|+ \ez_i|D\phi_{i,2}|\bigr)^{p-2}  |D\phi_{i,2}|^2\, dx\\
	\leq 
	C(n,p)\int_{\mathbb R^n}  \bigl(  |Dv|+ \ez_i|D\phi_i|\bigr)^{p-2}  |D\phi_i|^2\, dx\leq C(n,p).
	\end{multline}
	Hence, it follows by the analogue of \eqref{Lp phi} that
	\begin{equation}
	\label{eq:bound phi2}
	\int_{\mathbb R^n}  |D\phi_{i,1}|^p\,dx+\int_{\mathbb R^n}  |D\phi_{i,2}|^p\, dx\leq C(n,p).
	\end{equation}
	In particular we deduce that $\phi_{i,2} \rightharpoonup 0$ in $\dot W^{1,p}(\mathbb R^n)$ (as $|\{\ez_i \phi_i>\zeta v\}\cap B(0,R)|\to 0$ for any $R>1$)
	and that, up to a subsequence, both $\phi_i$ and $\phi_{i,1}$ converge weakly in $\dot W^{1,p}(\mathbb R^n)$ and also a.e. to the same function $\phi \in \dot W^{1,p}(\mathbb R^n)$.

	Let $\eta=\eta(n,p)>0$ be a small exponent to be fixed. We analyze two cases.

	\smallskip

\noindent 
{\it - Case 1.}
	If 	
	$$\int_{\{\ez_i \phi_i>\zeta v\}} |\phi_{i,1}|^{p^*}\, dx> \ez_i^{-\eta} \int_{\{\ez_i \phi_i>\zeta v\}}  \left(\phi_i-\frac {\zeta v}{\ez_i}\right)_+^{p^*}\, dx=\ez_i^{-\eta} \int_{\{\ez_i \phi_i>\zeta v\}} |\phi_{i,2}|^{p^*}\, dx,$$
	since $\phi$ is also the limit of $\phi_{i,1}$, 
	we can apply Step 1 to $\phi_{i,1}$ to deduce that  $\phi\in L^2(\R^n;v^{p^*-2})$ and
		$$ \int_{\mathbb R^n} \frac{(v+C_1\ez_i\phi_i)^{p^*}}{v^2+|\ez_i\phi_i|^2} |\phi_i|^2\, dx=\bigl(1+O(\ez_i^\eta)\bigr)\int_{\mathbb R^n}   \frac{(v+C_1\ez_i\phi_{i,1})^{p^*}}{v^2+|\ez_i\phi_{i,1}|^2} |\phi_{i,1}|^2\, dx \to \int_{\mathbb R^n}  v^{p^*-2} |\phi|^2\, dx,$$
	which proves \eqref{eq:L2 conv p small}.	
	
		\smallskip

\noindent 
{\it - Case 2.}
	Assume now that
	\begin{equation}\label{phi2 large}
	\int_{\{\ez_i \phi_i>\zeta v\}} |\phi_{i,1}|^{p^*}\, dx\le  \ez_i^{-\eta}  \int_{\{\ez_i \phi_i>\zeta v\}}  |\phi_{i,2}|^{p^*} \, dx.
	\end{equation}  
	We claim that
	\begin{equation}\label{phi2 inequ}
	\ez_i^{p^*-2}\int_{\mathbb R^n} |\phi_{i,2}|^{p^*} \, dx = O(\ez_i^{\eta}).
	\end{equation} 
	To prove this, denote $A_i:=\{\ez_i \phi_i>\zeta v\}$ and define
	$$E_i:=\left\{|D\phi_{i,2}|\le \frac {|Dv|}{\ez_i}\right\}\cap A_i,\qquad  F_i:=\left\{|D\phi_{i,2}|> \frac {|Dv|}{\ez_i}\right\}\cap A_i.$$
	Then, since $|Dv|+ \ez_i|D\phi_{i,2}| \leq 2|Dv|$ inside $E_i$,
	it follows by H\"older inequality that
		\begin{equation}
		\begin{split}
		 \label{two parts}
	&\int_{\mathbb R^n} |D\phi_{i,2}|^p\, dx = \int_{E_i} |D\phi_{i,2}|^p\, dx +\int_{F_i} |D\phi_{i,2}|^p\, dx \\
	&\qquad \leq \left( \int_{E_i} |Dv|^{p-2} |D\phi_{i,2}|^2\, dx\right)^{\frac p 2} \left(\int_{E_i} |Dv|^p\, dx\right)^{1-\frac p 2} +\int_{F_i} |D\phi_{i,2}|^p\, dx\\
		&\qquad \leq \left( 2^{2-p} \int_{E_i} \bigl(  |Dv|+ \ez_i|D\phi_{i,2}|\bigr)^{p-2} |D\phi_{i,2}|^2\, dx\right)^{\frac p 2} \left(\int_{E_i} |Dv|^p\, dx\right)^{1-\frac p 2} +\int_{F_i} |D\phi_{i,2}|^p\, dx\\
			&\qquad \leq C(n,p) \left( \int_{E_i} \bigl(  |Dv|+ \ez_i|D\phi_{i,2}|\bigr)^{p-2} |D\phi_{i,2}|^2\, dx\right)^{\frac p 2} \left(\int_{E_i} |Dv|^p\, dx\right)^{1-\frac p 2} +\int_{F_i} |D\phi_{i,2}|^p\, dx.
	\end{split}
	\end{equation}
Also, using \eqref{v and Dv} and \eqref{phi2 large} together with H\"older inequality (note that, since $1<p\leq \frac{2n}{n+2}$, we have $n\geq 3$) we get
	\begin{equation}
		\begin{split}
		 \label{two parts2}
	\int_{E_i} |Dv|^p\, dx &\le C(n,p) \int_{E_i} (1+|x|^{\frac p {p-1}})^{-n}|x|^{\frac p {p-1}}\, dx\\
	&\le C(n,p) \left(\int_{E_i}  \Big((1+|x|^{\frac p {p-1}})^{-n+1}|x|^{\frac p {p-1}}\Big)^{\frac {n}{n-2}} \, dx \right)^{1-\frac 2 n} \left(\int_{\R^n}  (1+|x|^{\frac p {p-1}})^{-\frac {n} {2}} \, dx \right)^{\frac {2}n}\\
	&\le C(n,p) \biggl(\int_{A_i}  \left(\frac{\ez_i\phi_i}{\zeta v}\right)^{p^*} \Big((1+|x|^{\frac p {p-1}})^{-n+1}|x|^{\frac p {p-1}}\Big)^{\frac {n}{n-2}} \, dx \biggr)^{1-\frac 2 n}\\
	&\le C(n,p)  \left( \ez_i^{p^*} \int_{A_i} |\phi_i|^{p^*}\, dx\right)^{{1-\frac 2 n}}\le C(n,p)  \left( \ez_i^{p^*-\eta} \int_{A_i} |\phi_{i,2}|^{p^*}\, dx\right)^{{1-\frac 2 n}},
		\end{split}
	\end{equation}
	where we used  that
	${\frac {np} {2(p-1)}}>n$ (since $p\leq \frac{2n}{n+2}<2$) and that 
	$$ v^{-p^*} \Big((1+|x|^{\frac p {p-1}})^{-n+1}|x|^{\frac p {p-1}}\Big)^{\frac {n}{n-2}} \le C(n,p).$$
	Therefore, introducing the notation
	$$
	N_{i,2}:=\int_{E_i} \bigl(  |Dv|+ \ez_i|D\phi_{i,2}|\bigr)^{p-2} |D\phi_{i,2}|^2\, dx,
	$$
	by Sobolev inequality,  \eqref{two parts}, and \eqref{two parts2}, we deduce that
	\begin{equation}
	\begin{split}\label{two parts 2}	 
	\ez_i^{p^*-2}\int_{\mathbb R^n} |\phi_{i,2}|^{p^*} \, dx &\le C(n,p)  \ez_i^{p^*-2}\left(\int_{\mathbb R^n} |D\phi_{i,2}|^{p} \, dx\right)^{\frac {p^*} p} \\
	&\le  C(n,p) \ez_i^{p^*-2}\biggl[ N_{i,2}^{\frac {p^*} 2}   \left(\int_{E_i} |Dv|^p\, dx\right)^{\frac {(2-p)p^*} {2p}} +\left( \int_{F_i} |D\phi_{i,2}|^p\, dx\right)^{\frac {p^*} p}\biggr]\\
	&\le C(n,p) \ez_i^{p^*-2}\biggl[ N_{i,2}^{\frac {p^*} 2}    \left(\ez_i^{p^*-\eta} \int_{A_i}  |\phi_{i,2}|^{p^*} \, dx\right)^{\frac {(2-p)(n-2)} {2(n-p)}} +  \int_{F_i} |D\phi_{i,2}|^p\, dx\biggr],
	\end{split}
	\end{equation}
	where in the last inequality we used \eqref{eq:bound phi2} and the fact that $\frac{p^*}{p}\geq 1$.
	
	Suppose first that 
	$$   \int_{F_i} |D\phi_{i,2}|^p\, dx\ge  N_{i,2}^{\frac {p^*} 2} \left(\ez_i^{p^*-\eta} \int_{A_i}  |\phi_{i,2}|^{p^*} \, dx\right)^{\frac {(2-p)(n-2)} {2(n-p)}}.$$
	Then, since $|Dv|\leq \ez_i |D\phi_{i,2}|\sim \ez_i|D\phi_{i}|$ inside $F_i$ (recall that $\zeta<1$), \eqref{eq:weighted W12} and \eqref{two parts 2}	yield
\begin{equation}
	\begin{split}
	\ez_i^{p^*-2}\int_{\mathbb R^n} |\phi_{i,2}|^{p^*} \, dx&\le C(n,p)\ez_i^{p^*-2}  \int_{F_i} |D\phi_{i,2}|^p\, dx
	  \\
	&= C(n,p)\ez_i^{p^*-p} \int_{F_i} \bigl( \ez_i |D\phi_{i,2}| \bigr)^{p-2}|D\phi_{i,2}|^2\, dx\\
	& \le C(n,p)\ez_i^{p^*-p} \int_{F_i} \bigl( |Dv|+\ez_i |D\phi_{i,2}| \bigr)^{p-2}|D\phi_{i,2}|^2\, dx , \label{phi21}
\end{split}
	\end{equation}
	which proves \eqref{phi2 inequ} choosing $\eta\leq p^*-p$ (recall \eqref{eq:bound phi12}).

	Consider instead the case
	$$   \int_{F_i} |D\phi_{i,2}|^p\, dx<   N_{i,2}^{\frac {p^*} 2}  \left(\ez_i^{p^*-\eta} \int_{A_i}  |\phi_{i,2}|^{p^*} \, dx\right)^{\frac {(2-p)(n-2)} {2(n-p)}},$$
	and set $\theta:=\frac {(2-p)(n-2)} {2(n-p)},$ so that \eqref{two parts 2}	yields
	\begin{multline*}
	\ez_i^{p^*-2}\int_{\mathbb R^n} |\phi_{i,2}|^{p^*} \, dx \leq C(n,p)\ez_i^{p^*-2} N_{i,2}^{\frac {p^*} 2} 
	\left(\ez_i^{p^*-\eta} \int_{A_i}  |\phi_{i,2}|^{p^*} \, dx\right)^{\theta}\\
	=C(n,p)\ez_i^{p^*-2+(2-\eta)\theta} N_{i,2}^{\frac {p^*} 2} \left(\ez_i^{p^*-2} \int_{A_i}  |\phi_{i,2}|^{p^*} \, dx\right)^{\theta}.
	\end{multline*}	 
	Since $\theta<1$, recalling the definition of $N_{i,2}$ and \eqref{eq:bound phi12}, this gives
	\begin{equation} \label{phi22}	\begin{split}
	\ez_i^{p^*-2}\int_{\mathbb R^n} |\phi_{i,2}|^{p^*} \, dx &\leq C(n,p)\ez_i^{\frac{p^*-2+(2-\eta)\theta}{1-\theta}} \biggl(\int_{E_i} \bigl(  |Dv|+ \ez_i|D\phi_{i,2}|\bigr)^{p-2} |D\phi_{i,2}|^2\, dx\biggr)^{\frac {p^*}{2(1-\theta)}}\\
	&\leq C(n,p)\ez_i^\eta \int_{E_i} \bigl(  |Dv|+ \ez_i|D\phi_{i,2}|\bigr)^{p-2} |D\phi_{i,2}|^2\, dx,
	\end{split} \end{equation}
	where the last inequality follows by choosing $\eta>0$ sufficiently small (notice that $p^*-2+2\theta>0$ and $\frac {p^*}{2(1-\theta)}>1$).
This proves \eqref{phi2 inequ} also in this case.
	
	\smallskip
	
		Now, combining \eqref{phi2 large} and \eqref{phi2 inequ}, we finally get
	\begin{multline*}
	\left|\int_{\mathbb R^n} \frac{(v+C_1\ez_i\phi_i)^{p^*}}{v^2+|\ez_i\phi_i|^2} |\phi_i|^2\, dx- \int_{\mathbb R^n}   \frac{(v+C_1\ez_i\phi_{i,1})^{p^*}}{v^2+|\ez_i\phi_{i,1}|^2} |\phi_{i,1}|^2\, dx\right|\\
	\le C(n,p,C_1) \left(\ez_i^{p^*-2} \int_{A_i}   |\phi_{i,2}|^{p^*}\, dx + \ez_i^2  \int_{A_i} \frac{(v+C_1\zeta v)^{p^*}}{v^2+|\zeta v|^2} |\zeta v|^2 \, dx\right)=o(1).
	\end{multline*}	 
	Thanks to this estimate, and since $\phi$ is also the limit of $\phi_{i,1}$, 
	applying
	Step 1 to $\phi_{i,1}$ we conclude the proof of the lemma.
\end{proof}
An important consequence of the proof of Lemma~\ref{weighted compact} is the following Orlicz-type Poincar\'e inequality: 

\begin{cor}
	\label{cor:weighted poincare}
	Let $1<p\le \frac {2n}{n+2}$.
	There exists $\ez_0=\ez_0(n,p)>0$ small such that the following holds: For any $\ez \in (0,\ez_0)$  and any $\phi\in \dot W^{1,p}(\mathbb R^n)\cap \dot W^{1,2}(\mathbb R^n;\, |Dv|^{p-2})$ with
	$$\int_{\mathbb R^n} \bigl( |Dv|+\ez|D\phi| \bigr)^{p-2} |D\phi|^2\, dx\le  1, $$
	we have
	\begin{equation}
	\label{weighted poincare}
	\int_{\mathbb R^n} (v+\ez \phi )^{p^*-2}|\phi|^2\, dx \le C(n,p) \int_{\mathbb R^n} \bigl( |Dv|+\ez|D\phi| \bigr)^{p-2} |D\phi|^2\, dx.
	\end{equation}
\end{cor} 
\begin{proof}
As in the proof of Lemma~\ref{weighted compact}, it suffices to consider the case $\phi \geq 0$.
Also, let $\zeta\in (0,1)$ be the small constant provided in the proof of Lemma~\ref{weighted compact}.

Write $\phi=\phi_{1}+\phi_{2}$, where 
	$$\phi_{1}:=\min\left\{\phi,\,\frac {\zeta v} {\ez}\right\},\qquad \phi_2:=\phi-\phi_1.$$
	Since $\ez\phi_1\le \zeta v$ we have $v\sim v+\ez\phi_1$, so 
	\eqref{weighted poincare} for $\phi_1$ follows from the analogue of  \eqref{poincare}.
	
	For $\phi_2$ we discuss two cases. 	
	
	\smallskip
		
	If 	
	$$\int_{\{\ez \phi>\zeta v\}} |\phi_1|^{p^*}\, dx> \int_{\{\ez \phi>\zeta v\}}  \left(\phi-\frac {\zeta v}{\ez}\right)_+^{p^*}\, dx= \int_{\{\ez \phi>\zeta v\}} |\phi_{2}|^{p^*}\, dx,$$
	then  
	$$ \int_{\mathbb R^n}  \ez^{p^*-2} |\phi_2|^{p^*}\, dx \le C(n,p) \int_{\mathbb R^n} v^{p^*-2}|\phi_1|^2\, dx.$$
	Thus, applying \eqref{weighted poincare} to $\phi_1$, we conclude that
	\begin{align*}
	\int_{\mathbb R^n} (v+\ez\phi)^{p^*-2}|\phi|^2\, dx
	&\le  C(n,p) \int_{\mathbb R^n} v^{p^*-2}|\phi_1|^2\, dx \\
	& \le C(n,p) \int_{\mathbb R^n} \bigl( |Dv|+\ez|D\phi_1| \bigr)^{p-2}|D\phi_1|^2\, dx\\
	& \le  C(n,p) \int_{\mathbb R^n} \bigl( |Dv|+\ez|D\phi| \bigr)^{p-2} |D\phi|^2 \, dx,
	\end{align*}
	where the last step follows from the analogue of \eqref{eq:bound phi12}.
		
	\smallskip
		
	On the other hand, when
		$$\int_{\{\ez \phi>\zeta v\}} |\phi_1|^{p^*}\, dx\leq \int_{\{\ez \phi>\zeta v\}} |\phi_{2}|^{p^*}\, dx,$$
		we can repeat the proofs of \eqref{phi21} and \eqref{phi22} with $\eta=0$ to deduce the validity of 
	\eqref{weighted poincare} for $\phi_2$.
		
	\smallskip
		
		Thus, by \eqref{eq:bound phi12} for $\phi$, and \eqref{weighted poincare} for $\phi_1$ and $\phi_2$, we eventually obtain 
	\begin{align*}
	&\int_{\mathbb R^n} (v+\ez \phi )^{p^*-2}|\phi|^2\, dx \le C(n,p) \int_{\mathbb R^n} v^{p^*-2}|\phi_1|^2\, dx+ C(n,p) \int_{\mathbb R^n}  \ez^{p^*-2} |\phi_2|^{p^*} \, dx\\
	&\qquad \le  C(n,p) \int_{\mathbb R^n} \bigl( |Dv|+\ez|D\phi_1| \bigr)^{p-2}|D\phi_1|^2\, dx+ C(n,p) \int_{\R^n}  \bigl( |Dv|+\ez|D\phi_2| \bigr)^{p-2}|D\phi_2|^2\, dx\\
	&\qquad \le  C(n,p) \int_{\mathbb R^n} \bigl( |Dv|+\ez|D\phi| \bigr)^{p-2}|D\phi|^2\, dx,
	\end{align*}
	which concludes the proof of the  corollary.	
\end{proof}

\subsection{Spectral gap}
Let $v=v_{a_0,1,0}$ be as in the previous section, and 
recall that
$$T_v\mathcal M:={\rm span}\left\{v,\,\partial_b v,\,\partial_{x_1}v,\,\dots ,\,\partial_{x_n}v\right\},$$
which is a subspace of $L^2(\mathbb R^n;v^{p^*-2})$.

In \cite[Proposition 3.1]{FN2019} it is proved that, for $p>2$, $T_v\mathcal M$ generates the first and the second eigenspaces corresponding to the linearized $p$-Laplacian operator
$$
\mathcal{L}_v[\varphi]:=-{\rm div}\left(|Dv|^{p-2}D\varphi+(p-2)|Dv|^{p-4}(Dv\cdot D\varphi) Dv \right)
$$
on the space $L^2(\mathbb R^n;v^{p^*-2})$ (this operator has a discrete spectrum for any $1<p<n$, thanks to Proposition~\ref{compact embedding}). 
We now note that, as a consequence of Proposition~\ref{compact embedding}, the proof of \cite[Proposition 3.1]{FN2019} extends to the full range $1<p<n$.\footnote{As noted 
in Remark~\ref{rem:weight}, for $1<p\le \frac {n+2}{n+1}$ the weight $|Dv|^{p-2}$ is not integrable at the origin, so one has to rely on singular  Sturm-Liouville theory (see for instance \cite[Chapter 8]{Z2005}).}
This shows that
$T_v\mathcal M$ generates the first and the second eigenspaces corresponding to $\mathcal{L}_v$, and therefore functions orthogonal to $T_v\mathcal M$ enjoy a quantitative improvement in the Poincar\'e inequality induced by $\mathcal{L}_v$. More precisely, the following holds:

\begin{prop}\label{spectrum gap}
Given $1<p<n$, and any function $\varphi\in L^2(\mathbb R^n;v^{p^*-2})$ orthogonal to $ T_v\mathcal M$, there exists a constant $\lambda=\lambda(n,p)>0$ so that
$$\int_{\mathbb R^n} |Dv|^{p-2} |D\varphi|^2+(p-2)|Dv|^{p-4}|Dv\cdot D\varphi|^2\,dx\ge \bigl((p^*-1)S^p+2\lambda\bigr)\|v\|_{L^{p^*}(\mathbb R^n)}^{p-p^*} \int_{\mathbb R^n} v^{p^*-2}|\varphi|^2\, dx,$$
where $S=S(n,p)$ is the optimal Sobolev constant.
\end{prop}

The following remark will be important to give a meaning to the notion of ``orthogonality to $T_v\mathcal M$'' for functions which are not necessarily in $L^2(\mathbb R^n;v^{p^*-2})$. 

\begin{rem}
\label{rmk:orth}
For any $\xi\in T_v\mathcal M$ it holds $v^{p^*-2}\xi \in L^{\frac{p^*}{p^*-1}}(\R^n)=\bigl(L^{p^*}(\R^n)\bigr)'$.
Hence, by abuse of notation, for any function $\psi \in L^{p^*}(\R^n)$
we say that $\psi$ is orthogonal to 
$T_v\mathcal M$ in $L^2(\mathbb R^n;v^{p^*-2})$ if
$$
\int_{\R^n}v^{p^*-2}\xi\,\psi\,dx=0\qquad \forall\, \xi \in T_v\mathcal M.
$$
\end{rem}
Note that, by H\"older inequality, $L^{p^*}(\R^n)\subset L^2(\R^n;v^{p^*-2})$ if $p^* \geq 2$. Hence, the notion of orthogonality introduced above is particularly relevant when $p^*<2$ (equivalently, $p<\frac{2n}{n+2}$).
We also observe that, by Sobolev embedding, the previous remark gives a meaning to the orthogonality to $T_v\mathcal M$ for functions in $\dot W^{1,p}(\mathbb R^n)$.

The main result of this section is the following spectral gap-type estimate.

\begin{prop}\label{new spectrum gap}
Let $S=S(n,p)$ be  the optimal Sobolev constant, and let $\lambda=\lambda(n,p)>0$ be as in Proposition \ref{spectrum gap}.
For any $\gamma_0>0$ and $C_1>0$ there exists $\bar\delta=\bar\delta(n,p,\gamma_0,C_1)>0$ such that the following holds:

Let $\varphi\in \dot W^{1,p}(\mathbb R^n)$ be orthogonal to $ T_v\mathcal M$ in $L^2(\mathbb R^n;v^{p^*-2})$, with
$$\|\varphi\|_{\dot W^{1,p}(\mathbb R^n)}\le \bar\delta.$$
Then:\\
\noindent (i) when  $1<p\le \frac {2n}{n+2}$, we have
\begin{multline*}
	\int_{\mathbb R^n} |Dv|^{p-2} |D\varphi|^2+(p-2)|w|^{p-2}\bigl( |D(v+\varphi)|-|Dv| \bigr)^2 + \gamma_0\min\bigl\{ |D\varphi|^{p},\, |Dv|^{p-2} |D\varphi|^2 \bigr\}  \,dx \\
	 \ge \bigl((p^*-1)S^p+\lambda\bigr) \|v\|_{L^{p^*}(\mathbb R^n)}^{p-p^*} \int_{\mathbb R^n} \frac {(v  +C_1|\varphi|)^{p^*}}{v^2+|\varphi|^2}|\varphi|^2\, dx,
\end{multline*}
where 
 $w:\R^n \to \R^n$ is defined in analogy to Lemma~\ref{vector inequ}: 
$$
w = \left\{ \begin{array}{cl}
\left(\frac {|D(v+\varphi)|}{(2-p)|D(v+\varphi)|+(p-1)|Dv|}\right)^{\frac 1 {p-2}} Dv & \textrm{if  $|Dv|<|D(v+\varphi)|$}\\
 Dv & \textrm{if  $|D(v+\varphi)|\le |Dv|$}
\end{array} \right.;
$$

\noindent (ii) when  $\frac {2n}{n+2}< p< 2$, we have
\begin{multline*}
	\int_{\mathbb R^n} |Dv|^{p-2} |D\varphi|^2+(p-2)|w|^{p-2}\bigl( |D(v+\varphi)|-|Dv| \bigr)^2 + \gamma_0\min\bigl\{ |D\varphi|^{p},\, |Dv|^{p-2} |D\varphi|^2 \bigr\}  \,dx \\
	 \ge \bigl((p^*-1)S^p+\lambda\bigr) \|v\|_{L^{p^*}(\mathbb R^n)}^{p-p^*} \int_{\mathbb R^n} v^{p^*-2}|\varphi|^2\, dx,
\end{multline*}
where 
 $w:\R^n \to \R^n$ is defined in analogy to Lemma~\ref{vector inequ}: 
$$
w = \left\{ \begin{array}{cl}
\left(\frac {|D(v+\varphi)|}{(2-p)|D(v+\varphi)|+(p-1)|Dv|}\right)^{\frac 1 {p-2}} Dv & \textrm{if  $|Dv|<|D(v+\varphi)|$}\\
 Dv & \textrm{if  $|D(v+\varphi)|\le |Dv|$}
\end{array} \right.;
$$

\noindent (iii)
when $p\geq 2$, we have
$$\int_{\mathbb R^n} |Dv|^{p-2} |D\varphi|^2+(p-2)|w|^{p-2}\bigl( |D(v+\varphi)|-|Dv| \bigr)^2 \,dx \ge \bigl((p^*-1)S^p+\lambda\bigr)\|v\|_{L^{p^*}(\mathbb R^n)}^{p-p^*} \int_{\mathbb R^n} v^{p^*-2}|\varphi|^2\, dx,$$
where  $w:\R^n \to \R^n$ is defined in analogy to Lemma~\ref{vector inequ}: 
$$
w = \left\{ \begin{array}{cl}
	Dv & \textrm{if  $|Dv|<|D(v+\varphi)|$}\\
	\left(\frac{|D(v+\varphi)|}{|Dv|}\right)^{\frac 1 {p-2}}D(v+\varphi) & \textrm{if  $|D(v+\varphi)|\le |Dv|$}
\end{array} \right..
$$
\end{prop}
\begin{proof}
We can assume that $\|v\|_{L^{p^*}(\mathbb R^n)}=1,$ as the general case follows by a scaling. 
Also, as in the proof of Lemma~\ref{weighted compact}, it suffices to consider the case $\varphi \geq 0$.

We argue by contradiction in all three cases. 

\smallskip

\noindent
$\bullet$ {\it The case $1< p\le \frac {2n}{n+2}$}.
Suppose the inequality does not hold. Then there exists  a sequence $0\not\equiv \varphi_i\to 0$ in $\dot W^{1,p}(\mathbb R^n)$, with $\varphi_i$ orthogonal to $ T_v\mathcal M$, such that
\begin{multline}
	\int_{\mathbb R^n} |Dv|^{p-2} |D\varphi_i|^2+(p-2)|w_i|^{p-2}\bigl(|D(v+\varphi_i)|-|Dv|\bigr)^2 + \gamma_0\min\bigl\{ |D\varphi_i|^{p},\, |Dv|^{p-2} |D\varphi_i|^2 \bigr\}  \,dx \\
	 < \bigl((p^*-1)S^p+\lambda\bigr) \int_{\mathbb R^n} \frac {(v +C_1\varphi_i)^{p^*}}{v^2+|\varphi_i|^2}|\varphi_i|^2\, dx,\label{eq:contradict 1 p}
\end{multline}
where $w_i$ corresponds to $\varphi_i$ as in  the statement.

Let $$
\ez_i:=\biggl(\int_{\mathbb R^n} \bigl(|Dv|+|D\varphi_i|\bigr)^{p-2} |D\varphi_i|^2 \, dx\biggr)^{\frac12},
$$
and set $\hat \varphi_i:= \frac{\varphi_i}{\ez_i}.$ 
Since $p<2$ it holds
$$\int_{\mathbb R^n} \bigl(|Dv|+|D\varphi_i|\bigr)^{p-2} |D\varphi_i|^2 \, dx\le 
\int_{\mathbb R^n} |D\varphi_i|^{p-2} |D\varphi_i|^2 \, dx=\int_{\mathbb R^n} |D \varphi_i|^p\, dx \to 0,$$
and hence $\ez_i\to 0. $

For any $R>1$, set
$$ \mathcal R_{i} :=  \{2 |Dv|\ge |D\varphi_i|\} ,\qquad \mathcal S_{i}:=\{2|Dv|< |D\varphi_i|\},
$$
$$ \mathcal R_{i,R}:=\bigl(B(0,R)\setminus B(0,1/R)\bigr) \cap \mathcal R_{i},\qquad \mathcal S_{i,R}:=\bigl(B(0,R)\setminus B(0,1/R)\bigr) \cap \mathcal S_{i}.$$
Since the integrand in the left hand side of \eqref{eq:contradict 1 p} is nonnegative (see \eqref{eq:lower G}),
we deduce that
\begin{multline}
\int_{B(0,R)\setminus B(0,1/R)} |Dv|^{p-2} |D\hat \varphi_i|^2+(p-2)|w_i|^{p-2}\left(\frac {|Dv +D\varphi_i| - |Dv|}{\ez_i} \right)^2 \\
+ \gamma_0\min\bigl\{ \ez_{i}^{p-2}|D\hat \varphi_i|^{p},\, |Dv|^{p-2} |D\hat \varphi_i|^2 \bigr\}\,dx
\le \bigl((p^*-1)S^p+\lambda\bigr) \int_{\mathbb R^n}  \frac {(v  +C_1 \varphi_i)^{p^*}}{v^2+|\varphi_i|^2}|\hat \varphi_i|^2\, dx \label{contradiction 1 p}
\end{multline}
for any $R>1$. 
Also, by \eqref{eq:lower G},
\begin{multline*}
|Dv|^{p-2} |D\hat \varphi_i|^2+(p-2)|w_i|^{p-2}\left(\frac {|Dv +D\varphi_i| - |Dv|}{\ez_i} \right)^2\\
\ge c(p)\frac{|Dv|}{|Dv|+|D \varphi_i|} |Dv|^{p-2}|D\hat \varphi_i|^2
\ge  c(p) |Dv|^{p-2}|D\hat \varphi_i|^2 \quad \text{on $\mathcal R_{i,R}$.}
\end{multline*} 
Thus, combining this bound with \eqref{contradiction 1 p}, we get
\begin{equation}
\begin{split}
&  c(p)\int_{\mathcal R_{i,R}} |Dv|^{p-2}|D\hat \varphi_i|^2\, dx +\gamma_0  \int_{\mathcal S_{i,R}}\ez_{i}^{p-2}|D\hat \varphi_i|^{p} \, dx  \\ 
&\qquad \le \int_{B(0,R)\setminus B(0,1/R)} |Dv|^{p-2} |D\hat \varphi_i|^2+(p-2)|w_i|^{p-2}\left(\frac {|Dv +D\varphi_i| - |Dv|}{\ez_i} \right)^2 \\ 
& \qquad \qquad \qquad \qquad\qquad \qquad \qquad + \gamma_0\min\bigl\{ \ez_{i}^{p-2}|D\hat \varphi_i|^{p},\, |Dv|^{p-2} |D\hat \varphi_i|^2 \bigr\}\,dx 
\\ 
&\qquad \le \bigl((p^*-1)S^p+\lambda\bigr) \int_{\mathbb R^n}  \frac {(v  +C_1 \varphi_i)^{p^*}}{v^2+|\varphi_i|^2}|\hat \varphi_i|^2\, dx. \label{RS 1 p}
\end{split}
\end{equation}
In particular, this implies that
\begin{multline}
\label{RS 1 p 2}
1=\ez_i^{-2}  \int_{\mathbb R^n} \bigl(|Dv|+|D\varphi_i|\bigr)^{p-2} |D \varphi_i|^2\, dx \\
\le C(p) \left[ \int_{\mathcal R_{i}} |Dv|^{p-2}|D\hat \varphi_i|^2\, dx +   \int_{\mathcal S_{i}}\ez_{i}^{p-2}|D\hat \varphi_i|^{p} \, dx \right]\\
\le C(n,p,\gamma_0) \bigl((p^*-1)S^p+\lambda\bigr) \int_{\mathbb R^n} \frac {(v  +C_1 \varphi_i)^{p^*}}{v^2+|\varphi_i|^2}|\hat \varphi_i|^2 \, dx.
\end{multline}
Furthermore, thanks to Corollary~\ref{cor:weighted poincare}, for $i$ large enough (so that $\ez_i\leq \ez_0$) we have
\begin{multline}
\label{eq:unif L2 weight}
\int_{\mathbb R^n}  \frac {(v  +C_1 \varphi_i)^{p^*}}{v^2+|\varphi_i|^2}|\hat \varphi_i|^2\, dx\le C(n,p,C_1) \int_{\mathbb R^n}   (v +|\varphi_i|)^{p^*-2} |\hat \varphi_i|^2\, dx\\
\le C(n,p,C_1)\int_{\mathbb R^n}   \bigl( |Dv| +|D\varphi_i| \bigr)^{p-2} |D\hat \varphi_i|^2\, dx\leq C(n,p,C_1).
\end{multline}
Hence, 
combining \eqref{RS 1 p} with \eqref{eq:unif L2 weight}, by the definition of $\mathcal S_{i,R}$ we get
$$
\ez_i^{-2}\int_{\mathcal S_{i,R}}|Dv|^p\,dx\leq \ez_i^{p-2}\int_{\mathcal S_{i,R}}|D\hat\varphi_i|^p\,dx\leq C(n,p,C_1),
$$
and since $|Dv|$ is uniformly bounded away from zero inside $B(0,R)\setminus B(0,1/R)$, 
we conclude that
\begin{equation}
\label{zero measure 1 p} 
|\mathcal S_{i,R}|\to 0  \qquad \text{as }i\to \infty,\quad \forall \,R>1.  
\end{equation}
Now, according to Lemma~\ref{weighted compact}, we have that
$\hat \varphi_i$ converges weakly in $\dot W^{1,p}(\R^n)$ to some function $\hat \varphi \in \dot W^{1,p}(\mathbb R^n)\cap L^2(\R^n,v^{p^*-2})$, and
\begin{equation}
\label{strong convergence}
\int_{\mathbb R^n} \frac {(v  +C_1 \varphi_i)^{p^*}}{v^2+|\varphi_i|^2}|\hat\varphi_i|^2\, dx\to \int_{\mathbb R^n} v^{p^*-2} |\hat \varphi|^2.
\end{equation}
Also, using again \eqref{RS 1 p} and \eqref{eq:unif L2 weight},
$$\int_{\mathcal R_{i,R}} |Dv|^{p-2}|D\hat \varphi_i|^2\, dx \leq C(n,p,C_1),$$
therefore \eqref{zero measure 1 p} and the weak convergence of $\hat\varphi_i$ to $\hat \varphi $ in $\dot W^{1,p}(\R^n)$ imply that, up to a subsequence,
$$D\hat \varphi_i\chi_{\mathcal R_{i,R}}\rightharpoonup D\hat \varphi\chi_{B(0,R)\setminus B(0,1/R)}\quad \text{in $L^2(\mathbb R^n,\mathbb R^n),$}\qquad \forall\,R>1. $$
In particular, $\hat\varphi \in \dot W^{1,2}_{\rm loc}(\R^n\setminus \{0\})$.
In addition, letting $i\to \infty$ in \eqref{RS 1 p 2} and \eqref{eq:unif L2 weight}, and using  \eqref{strong convergence}, we deduce that 
\begin{equation}
\label{nontrivial}0<c(n,p,\gamma_0)\leq \|\hat \varphi\|_{L^2(\mathbb R^n;v^{p^*-2})}\le C(n,p,C_1). 
\end{equation}
Let us write 
$$\hat \varphi_i =\hat \varphi +\psi_{i},\qquad \text{with }\psi_i:=\hat \varphi_i -\hat \varphi,$$
 so that
 $$
  \psi_{i}\rightharpoonup 0  \ \text{ in $\dot W^{1,p}(\R^n )$}\qquad\text{ and }  \qquad D\psi_{i} \chi_{\mathcal R_{i}}\rightharpoonup 0 \ \text{ in $L^2_{\rm loc}( \R^n\setminus\{0\},\R^n).$}
$$
We now look at the left hand side of \eqref{contradiction 1 p}.

The strong $\dot W^{1,p}(\mathbb R^n)$ convergence of $\varphi_i$ to $0$ implies that,
up to a subsequence, 
$|w_i| \to |Dv|$ a.e.
Also, we can rewrite 
\begin{multline*}
\left(\frac {|Dv +D\varphi_i| - |Dv|}{\ez_i} \right)^2=\biggl(\biggl[\int_0^1\frac{Dv+tD\varphi_i}{|Dv+tD\varphi_i|}\,dt\biggr]\cdot D\hat\varphi_i\biggr)^2\\
=\biggl(\biggl[\int_0^1\frac{Dv+tD\varphi_i}{|Dv+tD\varphi_i|}\,dt\biggr]\cdot \bigl(D\hat\varphi+D\psi_{i}\bigr)\biggr)^2.
\end{multline*}
Hence, if we set
$$
f_{i,1}:=\biggl[\int_0^1\frac{Dv+tD\varphi_i}{|Dv+tD\varphi_i|}\,dt\biggr]\cdot D\hat\varphi,\qquad f_{i,2}:=\biggl[\int_0^1\frac{Dv+tD\varphi_i}{|Dv+tD\varphi_i|}\,dt\biggr]\cdot D\psi_{i},
$$
since $\frac{Dv+tD\varphi_i}{|Dv+tD\varphi_i|} \to \frac{Dv}{|Dv|}$ a.e.,
it follows from Lebesgue's dominated convergence theorem  that
$$
f_{i,1} \to \frac{Dv}{|Dv|}\cdot D\hat \varphi \text{ strongly in $L^{2}_{\rm loc}(\R^n \setminus \{0\})$,}\qquad 
f_{i,2} \chi_{\mathcal R_{i}} \rightharpoonup 0 \text{ weakly in $L^{2}_{\rm loc}(\R^n \setminus \{0\})$.}
$$
Thus, we can control the first two terms of the left hand side of \eqref{contradiction 1 p} from below as follows:
\begin{equation}
\begin{split}
&\int_{\mathcal R_{i,R}} |Dv|^{p-2} |D\hat \varphi_i|^2+(p-2)|w_i|^{p-2}\left(\frac {|Dv +D\varphi_i| - |Dv|}{\ez_i} \right)^2\\
&\qquad= \int_{\mathcal R_{i,R}} |Dv|^{p-2} \Bigl( |D\hat\varphi|^2 +2D\psi_{i}\cdot D\hat\varphi \Bigr)+(p-2)|w_i|^{p-2}\left(f_{i,1}^2+2f_{i,1} f_{i,2}  \right)\\
&\qquad\qquad+\int_{\mathcal R_{i,R}} |Dv|^{p-2}  |D\psi_{i}|^2+(p-2)     |w_i|^{p-2}f_{i,2}^2\\
&\qquad\geq \int_{\mathcal R_{i,R}} |Dv|^{p-2} \Bigl( |D\hat\varphi|^2 +2D\psi_{i}\cdot D\hat\varphi \Bigr)+(p-2)|w_i|^{p-2}\left(f_{i,1}^2+2f_{i,1} f_{i,2}  \right),\label{eq:BR 1 p}
\end{split}
\end{equation}
where the last inequality follows from the nonnegativity of $|Dv|^{p-2}  |D\psi_{i}|^2+(p-2)     |w_i|^{p-2}f_{i,2}^2$ (thanks to \eqref{eq:lower G} and the fact that $f_{i,2}^2\leq  |D\psi_{i}|^2$).

Then, combining the convergences
$$
 D\psi_{i} \chi_{\mathcal R_{i}} \rightharpoonup 0,\quad f_{i,1}\to \frac{Dv}{|Dv|}\cdot D\hat \varphi,\quad
f_{i,2} \chi_{\mathcal R_{i}} \rightharpoonup 0 \qquad  \text{ in $L^{2}_{\rm loc}(\R^n \setminus \{0\})$,} $$
$$|w_i| \to |Dv|  \text { a.e.},\qquad  | \left(B(0,R)\setminus B(0,1/R)\right)\setminus \mathcal R_{i,R} | \to 0,
$$
with the fact that 
$$ |w_i|^{p-2}\le C(p) |Dv|^{p-2},$$
by Lebesgue's dominated convergence theorem we deduce that the last term in \eqref{eq:BR 1 p}  converges to 
$$
\int_{B(0,R)\setminus B(0,1/R)} |Dv|^{p-2} |D\hat \varphi|^2+(p-2)|Dv|^{p-2}\left(\frac{Dv}{|Dv|}\cdot D\hat \varphi\right)^2\,dx .
$$
Recalling \eqref{contradiction 1 p} and \eqref{strong convergence}, 
since $R>1$ is arbitrary and the integrand is nonnegative, this proves that
\begin{equation}
\label{eq:final}
\int_{\R^n} |Dv|^{p-2} |D\hat \varphi|^2+(p-2)|Dv|^{p-2}\left(\frac{Dv}{|Dv|}\cdot D\hat \varphi\right)^2\,dx \leq 
\bigl((p^*-1)S^p+\lambda\bigr) \int_{\mathbb R^n} v^{p^*-2}|\hat \varphi|^2\, dx.
\end{equation}
On the other hand, $\hat\varphi$ being the weak limit of $\hat\varphi_i$ in $\dot W^{1,p}(\R^n)$, it follows that $\hat\varphi_i\rightharpoonup \hat\varphi$ in $L^{p^*}(\R^n)$.
Hence, thanks to Remark~\ref{rmk:orth},
the orthogonality of $\varphi_i$ (and so of $\hat\varphi_i$) implies that also $\hat \varphi$ is orthogonal to $ T_v\mathcal M$. 
Since $\hat\varphi \in L^2(\mathbb R^n;v^{p^*-2})$, \eqref{nontrivial} and \eqref{eq:final} contradict Proposition \ref{spectrum gap}, concluding the proof.

\medskip

\noindent
$\bullet$ {\it The case $\frac {2n}{n+2}<  p<2$}.
The proof is very similar to the previous case, except for some small changes and a couple of different estimates.

If the statement fails, then there exists  a sequence $0\not\equiv \varphi_i\to 0$ in $\dot W^{1,p}(\mathbb R^n)$, with $\varphi_i$ orthogonal to $ T_v\mathcal M$, such that
\begin{multline}
\int_{\mathbb R^n} |Dv|^{p-2} |D\varphi_i|^2+(p-2)|w_i|^{p-2}\bigl(|D(v+\varphi_i)|-|Dv|\bigr)^2  + \gamma_0\min\bigl\{ |D\varphi_i|^{p},\, |Dv|^{p-2} |D\varphi_i|^2 \bigr\} \,dx\\
< \bigl((p^*-1)S^p+\lambda\bigr) \int_{\mathbb R^n} v^{p^*-2}|\varphi_i|^2\, dx,\label{eq:contradict 1p2}
\end{multline}
where $w_i$ corresponds to $\varphi_i$ as in  the statement.

As in the case $p\leq \frac{2n}{n+2}$,
we define
$$\ez_i:= \left( \int_{\mathbb R^n} \bigl(|Dv|+|D\varphi_i|\bigr)^{p-2} |D\varphi_i|^2\, dx \right)^{\frac 1 2},\qquad \hat \varphi_i=\frac {\varphi_i} {\ez_i},
$$
and we split $B(0,R)\setminus B(0,1/R)=\mathcal R_{i,R}\cup \mathcal S_{i,R}$.

Then, the analogues of \eqref{RS 1 p} and \eqref{RS 1 p 2} hold also in this case, with the only difference that the last term in both equations now becomes $\bigl((p^*-1)S^p+\lambda\bigr)\int_{\mathbb R^n} v^{p^*-2}|\hat\varphi_i|^2\, dx$.

We now observe that,
thanks to H\"older inequality, we have
\begin{multline*}
 \int_{\mathbb R^n}  |D\hat \varphi_i|^p\, dx\leq 
 \left( \int_{\mathbb R^n} \bigl(|Dv|+|D\varphi_i|\bigr)^{p-2}|D\hat \varphi_i|^2\, dx\right)^{\frac  p 2}  \left(\int_{\mathbb R^n} \bigl(|Dv|+|D\varphi_i|\bigr)^p dx\right)^{1-\frac  {p} 2} \\
 = \left(\int_{\mathbb R^n} \bigl(|Dv|+|D\varphi_i|\bigr)^p dx\right)^{1-\frac  {p} 2} \leq C(p)\biggl[\left(\int_{\mathbb R^n} |Dv|^p dx\right)^{1-\frac  {p} 2}+\ez_i^{\frac{p(2-p)}{2}} \left(\int_{\mathbb R^n} |D\hat\varphi_i|^p dx\right)^{1-\frac  {p} 2}\biggr],
\end{multline*}
from which it follows that
\begin{equation}
\label{uniform bound}
 \int_{\mathbb R^n}  |D\hat \varphi_i|^p\, dx\leq C(n,p).
\end{equation} 
Thus, 
up to a subsequence,  $\hat\varphi_i \to \hat \varphi $ weakly in $\dot W^{1,p}(\R^n)$
and
strongly in $L^2_{\rm loc}(\mathbb R^n)$ (note that now $p^*> 2$). In addition, H\"older and Sobolev inequalities, together with \eqref{uniform bound}, yield 
\begin{align*}
\int_{\mathbb R^n\setminus B(0,\rho)} v^{p^*-2}|\hat \varphi_i|^2\, dx & \le  \biggl(\int_{\mathbb R^n\setminus B(0,\,\rho)}v^{p^*}\,dx \biggr)^{1-\frac 2 {p^*}} \biggl(\int_{\mathbb R^n\setminus B(0,\,\rho)}|\hat \varphi_i|^{p^*}\,dx \biggr)^{\frac 2 {p^*}}\\
& \le  \biggl(\int_{\mathbb R^n\setminus B(0,\,\rho)}v^{p^*}\,dx \biggr)^{1-\frac 2 {p^*}} \left(\int_{\mathbb R^n}|D\hat \varphi_i|^{p}\,dx \right)^{\frac 2 {p}} \qquad \forall \,\rho\geq 0.
\end{align*}
Combining this bound with \eqref{uniform bound} and the strong convergence of $\hat\varphi_i$ to $\hat\varphi$ in $L^2_{\rm loc}(\R^n),$ 
we conclude that $\hat\varphi_i \to \hat \varphi $ strongly in $L^2(\mathbb R^n;\, v^{p^*-2})$. 

In particular, letting $i \to \infty$ in the analogue of \eqref{RS 1 p 2} we obtain
$$0<c(n,p)\leq \|\hat \varphi\|_{L^2(\mathbb R^n;v^{p^*-2})}. $$
Similarly, the analogue of \eqref{RS 1 p} implies that 
$$
|\mathcal S_{i,R}|\to 0\quad \text{and}\quad\int_{\mathcal R_{i,R}} |Dv|^{p-2}|D\hat \varphi_i|^2\, dx \leq C(n,p)  \qquad \forall \,R>1.
$$
So,  it follows from the weak convergence of $\hat\varphi_i$ to $\hat \varphi $ in $\dot W^{1,p}(\R^n)$  that,
up to a subsequence, 
$$D\hat \varphi_i\chi_{\mathcal R_{i,R}}\rightharpoonup D\hat \varphi\chi_{B(0,R)\setminus B(0,1/R)}\quad \text{ in $L^2(\mathbb R^n,\, \mathbb R^n)$},\qquad \forall\,R>1.$$
 Thanks to this bound, we can split
$$\hat \varphi_i =\hat \varphi +\psi_{i},\qquad \text{with }\psi_i:=\hat \varphi_i -\hat \varphi,$$
and the very same argument as in the case $p \leq \frac{2n}{n+2}$ allows us to deduce that
\begin{multline*}
\liminf_{i\to \infty}\int_{\mathcal R_{i,R}} |Dv|^{p-2} |D\hat \varphi_i|^2+(p-2)|w_i|^{p-2}\left(\frac {|Dv +D\varphi_i| - |Dv|}{\ez_i} \right)^2\\
\geq \int_{B(0,R)\setminus B(0,1/R)} |Dv|^{p-2} |D\hat \varphi|^2+(p-2)|Dv|^{p-2}\left(\frac{Dv}{|Dv|}\cdot D\hat \varphi\right)^2\,dx.
\end{multline*}
Recalling \eqref{eq:contradict 1p2}, 
since $R>1$ is arbitrary and the integrands above are nonnegative, this proves that
\eqref{eq:final} holds,
a contradiction to Proposition \ref{spectrum gap} since $\hat \varphi$ is orthogonal to $ T_v\mathcal M$ (being the strong $L^2(\R^n;v^{p^*-2})$-limit of $\hat \varphi_i$).

\medskip

\noindent
$\bullet$ {\it The case $p\geq 2$}. The argument is similar to the case $1<p<2$, but simpler.
 
If the statement of the lemma fails, then there exists  a sequence $0\not\equiv \varphi_i\to 0$ in $\dot W^{1,p}(\mathbb R^n)$, with $\varphi_i$ orthogonal to $ T_v\mathcal M$, such that
\begin{equation}
\label{eq:contradict}
\int_{\mathbb R^n} |Dv|^{p-2} |D\varphi_i|^2+(p-2)|w_i|^{p-2}\bigl(|D(v+\varphi_i)|-|Dv|\bigr)^2  \,dx< \bigl((p^*-1)S^p+\lambda\bigr) \int_{\mathbb R^n} v^{p^*-2}|\varphi_i|^2\, dx,
\end{equation}
where $w_i$ corresponds to $\varphi_i$ as in  the statement.

Let 
$$\ez_i:=\|\varphi_i\|_{\dot{W}^{1,2}(\mathbb R^n;|Dv|^{p-2})},\qquad \hat \varphi_i=\frac {\varphi_i} {\ez_i}.
$$
Note that, since $p\geq 2$, it follows by H\"older inequality that
$$\int_{\mathbb R^n} |Dv|^{p-2} |D\varphi_i|^2 \, dx\le
\biggl(\int_{\mathbb R^n} |D v|^{p}\, dx\biggr)^{1-\frac{p}{2}} \biggl(\int_{\mathbb R^n} |D \varphi_i|^p\, dx\biggr)^{\frac{p}{2}} \to 0,$$
hence $\ez_i\to 0. $

Since $1= \|\hat \varphi_i\|_{\dot{W}^{1,2}(\mathbb R^n;|Dv|^{p-2})}$,
Proposition~\ref{compact embedding} implies that,
up to a subsequence, $\hat\varphi_i \to \hat \varphi $ weakly in $\dot W^{1,2}_{\rm loc}(\R^n;\, |Dv|^{p-2})$
and
strongly in $L^2(\mathbb R^n;v^{p^*-2})$. 
Also, since $p\geq 2$, it follows from \eqref{eq:contradict} that
\begin{align*} 
1 = \int_{\mathbb R^n} |Dv|^{p-2} |D\hat \varphi_i|^2 \leq \bigl((p^*-1)S^p+\lambda\bigr) \int_{\mathbb R^n} v^{p^*-2}|\hat \varphi_i|^2\, dx,
\end{align*}
so we deduce that
$$\|\hat \varphi\|_{L^2(\mathbb R^n;v^{p^*-2})}\ge c(n,p)>0. $$
Also, since the integrand in the left hand side of \eqref{eq:contradict} is nonnegative,
we get
\begin{multline}
\int_{B(0,R)\setminus B(0,1/R)} |Dv|^{p-2} |D\hat \varphi_i|^2+(p-2)|w_i|^{p-2}\left(\frac {|Dv +D\varphi_i| - |Dv|}{\ez_i} \right)^2 \,dx
 \\ \le \bigl((p^*-1)S^p+\lambda\bigr) \int_{\mathbb R^n} v^{p^*-2}|\hat \varphi_i|^2\, dx\label{contradiction}
\end{multline}
for any $R>1$. 

Furthermore, because
$$
0< c(R)\leq |Dv|\leq C(R) \qquad \text{inside }B(0,R)\setminus B(0,1/R)\quad \forall\,R>1,
$$
writing
$$\hat \varphi_i =\hat \varphi +\psi_{i},\qquad \text{with }\psi_i:=\hat \varphi_i -\hat \varphi$$
we have
 $$
  \psi_{i} \rightharpoonup 0 \quad \text{ in $\dot W^{1,2}_{\rm loc}(\R^n \setminus \{0\})$.}
$$
Then we look at the left hand side of \eqref{contradiction}, and exactly as in the case $\frac{2n}{n+2}<p<2$
we deduce that
\begin{multline}
\label{eq:liminf R}
\liminf_{i\to \infty}\int_{B(0,R)\setminus B(0,1/R)} |Dv|^{p-2} |D\hat \varphi_i|^2+(p-2)|w_i|^{p-2}\left(\frac {|Dv +D\varphi_i| - |Dv|}{\ez_i} \right)^2\\
\geq \int_{B(0,R)\setminus B(0,1/R)} |Dv|^{p-2} |D\hat \varphi|^2+(p-2)|Dv|^{p-2}\left(\frac{Dv}{|Dv|}\cdot D\hat \varphi\right)^2\,dx.
\end{multline}
Recalling \eqref{contradiction} and
since $R>1$ is arbitrary,
this proves that \eqref{eq:final} holds, which contradicts Proposition \ref{spectrum gap} due to the orthogonality of $\hat\varphi$ to $T_v\mathcal M$.
\end{proof}

\section{Proof of Theorem~\ref{main thm}}
\label{sect:pf thm}

Thanks to the preliminary estimates performed in the previous sections, we can now follow the compactness strategy of \cite{BE1991,FN2019}.

By scaling, we can assume
$\|u\|_{L^{p^*}(\mathbb R^n)}=1$. Also, since the right hand side of \eqref{aim} is trivially bounded by $2$, it suffices to prove the result for $\delta(u)\ll 1$.

It follows by concentration-compactess that 
 for any $\hat \ez>0$ there exists a constant $\hat \delta=\hat\delta(n,p,\hat\ez)$ such that the following holds: if
$$
\|Du\|_{L^p(\mathbb R^n)}- S\le \hat\delta, $$
then there exists $\hat v\in \mathcal M$ which minimizes the right-hand side of \eqref{aim}, $\hat v$ satisfies $\frac 3 4\leq \|\hat v\|_{L^{p^*}}\leq \frac43$, and $\|Du-D\hat v\|_{L^{p}(\R^n)}\leq \hat\ez.$
Also, up to a translation and a rescaling that preserve the $L^{p^*}$-norm, we can assume that $\hat v=v_{a,1,0}$ with $a>0$.

As explained in the introduction, the basic idea would be to expand $u$ around $\hat v$. Unfortunately with this choice we would not have the desired orthogonality needed to use the spectral properties proved in the previous section.

Hence, we now prove the following result (recall Remark \ref{rmk:orth} for the notion of orthogonality when a function is in $L^{p^*}(\R^n)$): 
\begin{lem}
\label{lem:orthogonal}
Let $\|u\|_{L^{p^*}(\mathbb R^n)}=1$,
and assume that $\|Du-D\hat v\|_{L^{p}(\R^n)}\leq \hat\ez$ with $\hat v= v_{a,1,0} \in \mathcal M$.
There exist  $\ez'=\ez'(n,p)>0$ and a modulus of continuity $\omega:\R^+\to \R^+$ such that the following holds: If $\hat \ez\leq \ez'$, then
there exists $v \in \mathcal M$ such that $u-v$ is orthogonal to $T_v\mathcal M$ and $\|Du-Dv\|_{L^{p}(\R^n)}\leq \omega(\hat\ez).$
\end{lem}

\begin{proof}
Given $u$ as in the statement, we consider the minimization of the functional
\begin{equation}
\label{eq:min pb v}
\mathcal M\ni v\mapsto \mathcal F_u[v]:=\frac{1}{p^*}\int_{\R^n} |v|^{p^*}\,dx -  \frac{1}{p^*-1}\int_{\R^n} |v|^{p^*-2}v\,u\,dx.
\end{equation}
Assume first that $u=\hat v \in \mathcal M.$ We claim that the minimizer of \eqref{eq:min pb v} is unique and
coincides with $u$.

To prove this we 
note that, by H\"older inequality,
\begin{equation}
\label{eq:min F}
\mathcal F_u[v]
\geq \frac{1}{p^*}\int_{\R^n} |v|^{p^*}\,dx -  \frac{1}{p^*-1}\biggl(\int_{\R^n} |u|^{p^*}\,dx\biggr)^{\frac{1}{p^*}} \biggl(\int_{\R^n} |v|^{p^*}\,dx\biggr)^{\frac{p^*-1}{p^*}}
 \geq -\frac{1}{p^*(p^*-1)}\int_{\R^n} u^{p^*}\,dx,
\end{equation}
where the second inequality follows from the fact that the function $$(0,+\infty)\ni s \mapsto \frac{1}{p^*}s^{p^*}-\frac{1}{p^*-1}A\,s^{p^*-1}$$ is uniquely minimized at $s=A$.
Noticing that the last term in \eqref{eq:min F} coincides with $\mathcal F_u[u]$, and that equality holds in both inequalities of \eqref{eq:min F} if and only if $v=u$, the claim follows.

Now, if $u$ is close to $\hat v=v_{a,1,0}$ in $\dot W^{1,p}(\R^n)$-norm, it follows by compactness that the minimum of the function
$$
\R\times (0,+\infty)\times \R^n \ni (a,b,x_0)\mapsto \mathcal F_u[v_{a,b,x_0}]
$$
is attained at some values $(a',b',x_0')$
close to $(a,1,0)$, hence $\|Dv_{a',b',x_0'}-D\hat v\|_{L^p(\R^n)} \ll 1$. Thus, since by assumption $u$ and $\hat v$ are $\dot W^{1,p}(\R^n)$-close, we deduce that
$$
\|Du-Dv_{a',b',x_0'}\|_{L^{p}(\R^n)} \to 0 \qquad \text{as }\|Du-D\hat v\|_{L^{p}(\R^n)} \to 0,
$$
which proves the existence of a modulus of continuity $\omega$ as in the statement.

Finally, it is immediately to check that if $v\in \mathcal M$ is close to $v_{a,1,0}$ and minimizes $\mathcal F_u$, 
then
$$
0=\frac{d}{dt}\Big|_{t=0}\mathcal F_u[v+t\xi]=\int_{\R^n}v^{p^*-2}\xi\,(v-u) \,dx\qquad \forall\,\xi \in T_v\mathcal M.
$$
This concludes the proof.
\end{proof}

Thanks  to Lemma \ref{lem:orthogonal}, given $u$ as at the beginning of the section with $\delta(u)$ sufficiently small, we can find $v \in \mathcal M$ close to $u$ such that $u-v$ is orthogonal to $T_v\mathcal M$.
More precisely, $u$ can be written as $u=v+\ez\varphi$, where $\ez\leq \omega(\hat\ez)$ with $\hat\ez\leq  \ez'$, $\| D\varphi \|_{L^{p}(\mathbb R^n)}=1$, and  $\varphi$ is orthogonal to $T_v\mathcal M$ (see Remark \ref{rmk:orth}).
Furthermore,
up to a further small translation and rescaling, we can assume that $v=v_{a_0,1,0}$ with $\frac12 \leq \|v\|_{L^{p^*}}\leq 2$ so that all the statements in Section~\ref{spectrum gaps} hold.

Observe that, for $\delta(u)$ small,
\begin{equation}\label{lower estimate}
 \delta(u)=\|Du\|_{L^p(\mathbb R^n)}- S \ge c(n,p) \Big(\|Du\|^p_{L^p(\mathbb R^n)} -  S^p\Big).
\end{equation}
In the following argument several parameters will appear, and these parameters depend on each other.
To simplify the notation we shall not explicit their dependence on $n$ and $p$, but we emphasize how the parameters depend on each other, at least until they have been fixed.

\smallskip
\noindent
$\bullet$ {\it The case $1 <p\leq \frac {2n} {n+2}$}.
Let $\kappa>0$ be a small constant to be fixed later.
By  Lemma~\ref{vector inequ} we have
\begin{align*}
  \|Du\|^p_{L^p(\mathbb R^n)}&=  \int_{\mathbb R^n}|Dv+\ez D\varphi|^p \,dx  \\
& \ge  \int_{\mathbb R^n} |Dv|^p\,dx+\ez p \int_{\mathbb R^n} |Dv|^{p-2}Dv\cdot D\varphi \,dx  \\
&  \qquad \qquad    +    \frac{\ez^2p(1-\kappa)} 2  \biggl(\int_{\mathbb R^n} |Dv|^{p-2}|D\varphi|^2 + (p-2) |w|^{p-2}\biggl(\frac{|Du|-|Dv|}{\ez}\biggr)^2 \,dx \biggr)\\
& \qquad \qquad   \qquad \qquad  \qquad \qquad  + c_0(\kappa) \int_{\mathbb R^n} \min\bigl\{\ez^p |D\varphi|^p,\,\ez^2|Dv|^{p-2} |D\varphi|^2\bigr\} \,dx,
\end{align*}
where $w$ corresponds to $u$ and $v$ as in Lemma~\ref{vector inequ}.
On the other hand, by Lemma~\ref{upper bound} and the concavity of $t\mapsto t^{\frac p {p^*}}$,
\begin{align*}
1&=\|u\|^p_{L^{p^*}(\mathbb R^n)}=\left(\int_{\mathbb R^n}|v+\ez \varphi|^{p^*} \,dx\right)^{\frac p {p^*}}\\
& \le  \|v\|^p_{L^{p^*}(\mathbb R^n)} + \|v\|_{L^{p^*}(\mathbb R^n)}^{p-p^*}\left(\ez p \int_{\mathbb R^n} v^{p^*-1}\varphi\,dx +\ez^2\left(\frac{ p(p^*-1)} 2 +\frac{p \kappa}{p^*}\right)\int_{\mathbb R^n}\frac {(v  +C_1(\kappa)|\ez\varphi|)^{p^*}}{v^2+|\ez \varphi|^2}|\varphi|^2\,dx\right). 
\end{align*}
Since, by \eqref{S equ}, 
$$\ez p \int_{\mathbb R^n} |Dv|^{p-2}Dv\cdot D\varphi \,dx=\|v\|_{L^{p^*}(\mathbb R^n)}^{p-p^*} S^p \ez p  \int_{\mathbb R^n} v^{p^*-1}\varphi\,dx,$$
and $\|Dv\|_{L^p(\mathbb R^n)}=S \|v\|_{L^{p^*}(\mathbb R^n)}$, we then immediately conclude that
\begin{multline*}
 C(n,p) \delta(u)\ge   \|Du\|^p_{L^p(\mathbb R^n)} -  S^p \|u\|^p_{L^{p^*}(\mathbb R^n)} \\
\ge    \frac{\ez^2 p(1-\kappa) } 2  \biggl(\int_{\mathbb R^n} |Dv|^{p-2}|D\varphi|^2 + (p-2) |w|^{p-2}\biggl(\frac{|Du|-|Dv|}{\ez}\biggr)^2 \,dx \biggr)\\
 \ + c_0(\kappa) \int_{\mathbb R^n}  \min\bigl\{ \ez^p|D\varphi|^p,\,\ez^2|Dv|^{p-2}  |D\varphi|^2\bigr\} \,dx \\
 - \ez^2 \|v\|_{L^{p^*}(\mathbb R^n)}^{p-p^*} S^p\left(\frac{p (p^*-1)} 2 +\frac{p \kappa}{p^*}\right)\int_{\mathbb R^n}\frac {(v  +C_1(\kappa)|\ez\varphi|)^{p^*}}{v^2+|\ez\varphi|^2}|\varphi|^2\,dx.
\end{multline*}
Now, for $\delta(u) \leq \delta'=\delta'(\ez,\kappa,\gamma_0)$ small enough,  Proposition~\ref{new spectrum gap} allows us to  reabsorb the last term above: more precisely, we have
\begin{align*}
&C(n,p) \delta(u) \\
&\qquad \ge  p\ez^2  \biggl(\frac{(1-\kappa)} 2 - \frac{(p^*-1)+\frac2{p^*}\kappa }{2(p^*-1)+2\lambda S^{-p}} \biggr)  \biggl(\int_{\mathbb R^n} |Dv|^{p-2}|D\varphi|^2 + (p-2) |w|^{p-2}\biggl(\frac{|Du|-|Dv|}{\ez}\biggr)^2 \,dx \biggr)  \\
& \qquad +   \biggl(c_0(\kappa) - \gamma_0\frac{p\bigl[(p^*-1)+\frac2{p^*}\kappa \bigr]}{2(p^*-1)+2\lambda S^{-p}}  \biggr) \int_{\mathbb R^n} \min\bigl\{ \ez^p|D\varphi|^p,\,\ez^2|Dv|^{p-2}  |D\varphi|^2\bigr\} \,dx,
\end{align*}
and choosing first $\kappa=\kappa(n,p)>0$ small  enough so that
$$\frac{(1-\kappa)} 2 - \frac{(p^*-1)+\frac2{p^*}\kappa }{2(p^*-1)+2\lambda S^{-p}}\geq 0,
$$
and then $\gamma_0=\gamma_0(n,p)>0$  small  enough so that
$$ \frac {c_0} 2 \geq  \gamma_0\frac{p\bigl[(p^*-1)+\frac2{p^*}\kappa \bigr]}{2(p^*-1)+2\lambda S^{-p}},$$
we eventually arrive at
\begin{align}
C(n,p) \delta(u) 
\ge   \frac {c_0} 2  \int_{\mathbb R^n} \min\bigl\{ \ez^p|D\varphi|^p,\,\ez^2|Dv|^{p-2}  |D\varphi|^2\bigr\} \,dx. \label{lowerbd} 
\end{align}
Observe that, since $p<2$, it follows by H\"older inequality that
\begin{align*}
 \biggl( \int_{\{\ez |D\varphi|< |Dv|\}}   |D\varphi|^p \,dx  \biggr)^{\frac 2 p }&\leq 
\biggl(\int_{\{\ez |D\varphi|< |Dv|\}}   |Dv|^p \,dx \biggr)^{\frac 2 p -1}
\int_{\{\ez |D\varphi|< |Dv|\}}  |Dv|^{p-2}  |D\varphi|^2 \,dx\\
&\leq C(n,p)\int_{\{\ez |D\varphi|< |Dv|\}}  |Dv|^{p-2}  |D\varphi|^2 \,dx.
\end{align*}
Hence, since $\| D\varphi \|_{L^{p}(\mathbb R^n)}=1$, we get
\begin{align}
&\int_{\mathbb R^n} \min\bigl\{ \ez^p|D\varphi|^p,\,\ez^2|Dv|^{p-2}  |D\varphi|^2\bigr\} \,dx \nonumber  \\
 &\qquad = \int_{\{\ez |D\varphi|\ge |Dv|\}}   \ez^p|D\varphi|^p  \,dx + \int_{\{\ez |D\varphi|< |Dv|\}}  \ez^2|Dv|^{p-2}  |D\varphi|^2 \,dx \nonumber \\
 & \qquad \ge \int_{\{\ez |D\varphi|\ge |Dv|\}}   \ez^p|D\varphi|^p  \,dx + c\biggl(\int_{\{\ez |D\varphi|< |Dv|\}}   \ez^p  |D\varphi|^p \,dx \biggr)^{\frac 2 p } \ge c   \left(\int_{\mathbb R^n} \ez^p |D\varphi|^p \,dx \right)^{\frac 2 p }, \label{holder}
\end{align}
where $c=c(n,p)>0$.
 
Combining  \eqref{lowerbd} and \eqref{holder},
we conclude the proof of \eqref{aim} with $\alpha=2$.
\medskip

\noindent
$\bullet$ {\it The case $\frac {2n} {n+2}< p<2$}. The proof is very similar to the previous case, with very small changes.

By Lemma~\ref{vector inequ}  we have
\begin{multline*}
\int_{\mathbb R^n}|Du|^p \,dx- \int_{\mathbb R^n} |Dv|^p\,dx-\ez p \int_{\mathbb R^n} |Dv|^{p-2}Dv\cdot D\varphi \,dx\\
\ge     \frac{\ez^2p(1-\kappa)} 2  \biggl(\int_{\mathbb R^n} |Dv|^{p-2}|D\varphi|^2 + (p-2) |w|^{p-2}\biggl(\frac{|Du|-|Dv|}{\ez}\biggr)^2 \,dx \biggr)\\
+ c_0(\kappa)\int_{\mathbb R^n} \min\bigl\{ \ez^p|D\varphi|^p,\,\ez^2|Dv|^{p-2}  |D\varphi|^2\bigr\} \,dx,
\end{multline*}
where $w$ corresponds to $u$ and $v$ as in Lemma~\ref{vector inequ}, while by Lemma~\ref{upper bound}
\begin{align*}
\int_{\mathbb R^n}|u|^{p^*} \,dx\le  1+\ez p^*\int_{\mathbb R^n} v^{p^*-1}\varphi\,dx +\ez^2\left(\frac{p^*(p^*-1)} 2 +\kappa\right)\int_{\mathbb R^n}v^{p^*-2}|\varphi|^2\,dx+ \ez^{p^*} C_1(\kappa)\int_{\mathbb R^n}|\varphi|^{p^*} \,dx. 
\end{align*}
Hence, arguing as in the case $1<p\leq \frac{2n}{n+2}$, it follows from \eqref{S equ}, Proposition~\ref{new spectrum gap}, and \eqref{holder} that, by choosing  first $\kappa>0$ and then $\gamma_0>0$ small enough, for $\delta(u)$ sufficiently small we have
\begin{align*}
\int_{\mathbb R^n}|Du|^p \,dx- \int_{\mathbb R^n} |Dv|^p\,dx
\ge c   \left(\int_{\mathbb R^n} \ez^p |D\varphi|^p \,dx \right)^{\frac 2 p } - \ez^{p^*} \frac{C_1 p}{p^*} \int_{\mathbb R^n}|\varphi|^{p^*} \,dx.
\end{align*}
Since $p^*>2$ and $1=\|D\varphi\|_{L^p(\R^n)}\geq S\|\varphi\|_{L^{p^*}(\R^n)}$, the result follows by the Sobolev inequality, provided $\ez$ is sufficiently small.

\medskip

\noindent
$\bullet$ {\it The case $p\geq 2$}. 
By Lemma~\ref{vector inequ}  we have
\begin{multline*}
\int_{\mathbb R^n}|Du|^p \,dx- \int_{\mathbb R^n} |Dv|^p\,dx-\ez p \int_{\mathbb R^n} |Dv|^{p-2}Dv\cdot D\varphi \,dx\\
\ge     \frac{\ez^2p(1-\kappa)} 2  \biggl(\int_{\mathbb R^n} |Dv|^{p-2}|D\varphi|^2 + (p-2) |w|^{p-2}\biggl(\frac{|Du|-|Dv|}{\ez}\biggr)^2 \,dx \biggr)+\ez^p c_0(\kappa)\int_{\mathbb R^n} |D\varphi|^p \,dx,
\end{multline*}
where $w$ corresponds to $u$ and $v$ as in Lemma~\ref{vector inequ}, while by Lemma~\ref{upper bound}
\begin{align*}
\int_{\mathbb R^n}|u|^{p^*} \,dx\le  1+\ez p^*\int_{\mathbb R^n} v^{p^*-1}\varphi\,dx +\ez^2\left(\frac{p^*(p^*-1)} 2 +\kappa\right)\int_{\mathbb R^n}v^{p^*-2}|\varphi|^2\,dx+ \ez^{p^*} C_1(\kappa)\int_{\mathbb R^n}|\varphi|^{p^*} \,dx. 
\end{align*}
Hence, arguing again as in the case $p\leq \frac{2n}{n+2}$, it follows from \eqref{S equ} and Proposition~\ref{new spectrum gap} that, by choosing  $\kappa>0$ small enough,
\begin{align*}
\int_{\mathbb R^n}|Du|^p \,dx- \int_{\mathbb R^n} |Dv|^p\,dx
\ge \ez^p c_0\int_{\mathbb R^n}  |D\varphi|^p \,dx - \ez^{p^*} \frac{C_1p}{p^*} \int_{\mathbb R^n}|\varphi|^{p^*} \,dx.
\end{align*}
Since $1=\|D\varphi\|_{L^p(\R^n)}\geq S\|\varphi\|_{L^{p^*}(\R^n)}$, this implies \eqref{aim} with $\alpha=p$ when $\ez$ is sufficiently small, concluding the proof of Theorem~\ref{main thm}.
\qed

\appendix

\section{A Hardy-Poincare inequality}
\begin{lem}\label{HP ineq}
Let $\alpha< n$ and let $u\in \dot W^{1,\,p}\bigl(\mathbb R^n; |x|^{-\az}\bigr)$. Then, for any $R>1$, we have
$$\int_{\mathbb R^n\setminus {B(0,R)}} |u|^p  |x|^{-\az}\, dx \le C(n,p,\az) \int_{\mathbb R^n\setminus {B(0,R)}}|Du|^p  |x|^{-\az+p} \, dx.$$
\end{lem}
\begin{proof}
Since $R\geq 1$ and $\alpha<n$, thanks to Fubini's Theorem and using polar coordinates we get
\begin{align*}
&\int_{\mathbb R^n\setminus {B(0,R)}} |u|^p  |x|^{-\az}\, dx \\
&\qquad \le    C(n,p)  \int_{\mathbb S^{n-1}} \int_{R}^{\infty} |u(r\theta)|^p  r^{-\az+n-1 }\, dr \, d\theta \\
&\qquad \le   C(n,p)  \int_{\mathbb S^{n-1}} \int_{R}^{\infty} \int_{r}^{\infty}|u(t\theta)|^{p-1} |Du|(t\theta)  r^{-\az+n-1 }\, dt\, dr\, d\theta \\
 &\qquad \le  C(n,p)  \int_{\mathbb S^{n-1}} \int_{R}^{\infty} \int_{1}^{t}|u(t\theta)|^{p-1} |Du|(t\theta)  r^{-\az+n-1 }\, dr\, dt\, d\theta \\
&\qquad \le  C(n,p,\az)  \int_{\mathbb S^{n-1}} \int_{R}^{\infty} |u(t\theta)|^{p-1} |Du|(t\theta)  t^{-\az+n} \, dt\, d\theta\\
&\qquad \le  C(n,p,\az) \left( \int_{\mathbb S^{n-1}}\int_{R}^{\infty} |u(t\theta)|^{p}  t^{-\az+n-1} \, dt\, d\theta \right)^{\frac {p-1} p}\cdot \\
&\qquad\qquad\qquad\qquad\qquad \qquad \cdot\left(\int_{\mathbb S^{n-1}} \int_{R}^{\infty} |Du|^p(t\theta)  t^{-\az+n-1+p} \, dt\, d\theta \right)^{\frac 1 p}\\
  &\qquad \le  C(n,p,\az)  \biggl(   \int_{\mathbb R^n\setminus {B(0,R)}}  |u(x)|^{p}  |x|^{-\az} \, dx  \biggr)^{\frac {p-1} p}  \biggl(  \int_{\mathbb R^n\setminus {B(0,R)}} |Du|^p(x)  |x|^{-\az+p} \, dx \biggr)^{\frac 1 p}
\end{align*}
where we applied H\"older inequality in the penultimate step. This implies the lemma. 
\end{proof}

\section{A numerical inequality}

\begin{lem}
Let $1<p\leq \frac{2n}{n+2}$.
Given $\ez_0>0$, there exists $\zeta=\zeta(\ez_0)$ small enough so that
the following inequality holds for any nonnegative numbers $\epsilon,r,\,a,\,b$  satisfying $\ez \in (0,1)$ and $\ez a\le \zeta \bigl(1+r^{\frac p {p-1}}\bigr)^{1-\frac n p}$:
	\begin{align}
	&\bigl(1+r^{\frac p {p-1}}\bigr)^{\left(1-\frac n p\right)(p^*-2)+p-1}\left[a^2 \zeta^p r^{\frac p {p-1}}\bigl(1+r^{\frac p {p-1}}\bigr)^{-p}+ a^2\ez^p b^p \bigl(1+r^{\frac p {p-1}}\bigr)^{n-p}+a^{2-p} b^p\right] \nonumber \\
	&\qquad\le  \ez_0 \bigl(1+r^{\frac p {p-1}}\bigr)^{\left(1-\frac n p\right)(p^*-2)} a^2+ C(\ez_0,n,p)(1+r )^{-\frac p {p-1}} \left(\bigl(1+r^{\frac p {p-1}}\bigr)^{-\frac n p}  r^{\frac 1 {p-1}}+\ez b\right)^{p-2} b^2 \label{inter} \\
	&\qquad\le   \ez_0 \bigl(1+r^{\frac p {p-1}}\bigr)^{\left(1-\frac n p\right)(p^*-2)} a^2+ C(\ez_0,n,p) \left(\bigl(1+r^{\frac p {p-1}}\bigr)^{-\frac n p}  r^{\frac 1 {p-1}}+\ez b\right)^{p-2} b^2.	\label{young}
	\end{align}\end{lem}

\begin{proof}
Note that \eqref{young} immediately follows from \eqref{inter}, so it suffices to prove \eqref{inter}.
We distinguish several cases.

	\smallskip

\noindent $\bullet$ {\it Case 1:} $0 \leq r \leq 1$.
	In this case,  up to changing the values of $\ez_0$ and $\zeta$ by a universal constant, \eqref{inter} is  equivalent to 
	\begin{equation}
	\label{young r small}
	a^2 \zeta^p r^{\frac p {p-1}}+a^2\ez^p b^p + a^{2-p} b^p\le \ez_0 a^2 + C(\ez_0,n,p)\bigl(r^{\frac 1 {p-1}}+\ez b\bigr)^{p-2}b^2.
	\end{equation}
	 Note that:\\
	- if 
	$\ez b \le \left(\frac{\ez_0}{3}\right)^{\frac 1 p} r^{\frac 1 {p-1}}$
	then $a^2\ez^p b^p\le \frac {\ez_0} 3 a^2$;\\
- if 
	$\ez b > \left(\frac{\ez_0}{3}\right)^{\frac 1 p} r^{\frac 1 {p-1}}$
	then, since $\ez a\le \zeta \bigl(1+r^{\frac p {p-1}}\bigr)^{1-\frac n p}\leq 2\zeta$,  
	$$a^2\ez^p b^p\le 4\zeta^2 \ez^{p-2} b^p\le C(\ez_0,n,p)\bigl(r^{\frac 1 {p-1}}+\ez b\bigr)^{p-2}b^2.$$
	Similarly:\\
	- if 
	$b\le \left(\frac {\ez_0} 3\right)^{\frac 1 p} a$
	then $a^{2-p}b^p\le \frac {\ez_0} 3 a^2$;\\
	- if $ \left(\frac {\ez_0} 3\right)^{\frac 1 p} a<b<\ez^{-1} r^{\frac 1 {p-1}}$ then
	$$a^{2-p}b^p\le C(\ez_0,n,p) b^2\le C(\ez_0,n,p) r^{\frac {p-2}{p-1}}b^2\le C(\ez_0,n,p)\bigl(r^{\frac 1 {p-1}}+\ez b\bigr)^{p-2}b^2; $$
- if $b\ge \ez^{-1} r^{\frac 1 {p-1}}$ then, since $\ez a\le \zeta \bigl(1+r^{\frac p {p-1}}\bigr)^{1-\frac n p}\le 2\zeta$, 
	$$a^{2-p}b^p\le 4^{2-p} \zeta^{2-p} \ez^{p-2} b^p\le C(\ez_0,n,p)\bigl(r^{\frac 1 {p-1}}+\ez b\bigr)^{p-2}b^2.$$
	Thus, choosing $\zeta^p\le \frac {\ez_0}{3}$,  \eqref{young r small} holds in all cases.

	\smallskip
	
	\noindent $\bullet$ {\it Case 2:} $ r> 1$.
	In this case, \eqref{inter} is equivalent to
	\begin{multline}\label{young r large}
	r^{\frac{p-n}{p-1}(p^*-2)}a^2\zeta^p+ r^{\frac{p-n}{p-1}(p^*-2-p)+p}a^p \ez^p b^p+a^{2-p}b^pr^{\frac{p-n}{p-1}(p^*-2)+p}\\
	\le \ez_0 r^{\frac{p-n}{p-1}(p^*-2)}a^2 + C(\ez_0,n,p)r^{-\frac p {p-1}}\bigl(r^{\frac {1-n} {p-1}}+\ez b\bigr)^{p-2}b^2. 
	\end{multline}
	Again:\\
	- if $b\le \left(\frac {\ez_0} 3\right)^{\frac 1 p} r^{\frac{1-n}{p-1}}\ez^{-1}$ then
	$$r^{\frac{p-n}{p-1}(p^*-2-p)+p}a^2 \ez^p b^p\le \frac {\ez_0} 3 r^{\frac{p-n}{p-1}(p^*-2)}a^2;$$
	-if $b> \left(\frac {\ez_0} 3\right)^{\frac 1 p} r^{\frac{1-n}{p-1}}\ez^{-1}$, we apply the inequality
	$\ez a\le \zeta \bigl(1+r^{\frac p {p-1}}\bigr)^{1-\frac n p}\leq 2\zeta r^{\frac {p-n} {p-1}}$ to conclude
	$$r^{\frac{p-n}{p-1}(p^*-2-p)+p}a^2 \ez^p b^p\le 4 r^{-\frac{p }{p-1}}  \zeta^2 \ez^{p-2}  b^p\le C(\ez_0,n,p)r^{-\frac{p }{p-1}}\bigl(r^{\frac {1-n} {p-1}}+\ez b\bigr)^{p-2}b^2.$$
	On the other hand:\\
	- if $b\le \left(\frac {\ez_0} 3\right)^{\frac 1 p} a r^{-1}$ then
	$$a^{2-p}b^pr^{\frac{p-n}{p-1}(p^*-2)+p}\le \frac{\ez_0} 3 r^{\frac{p-n}{p-1}(p^*-2)}a^2;$$
   - if $\left(\frac {\ez_0} 3\right)^{\frac 1 p} a r^{-1} <b< \ez^{-1} r^{\frac {1-n} {p-1}}$  then
   	\begin{multline*}
	a^{2-p}b^p r^{\frac{p-n}{p-1}(p^*-2)+p}\le C(\ez_0,n,p) b^2 r^{\frac{p-n}{p-1}(p^*-2)+2}\\
	= C(\ez_0,n,p) r^{-\frac p {p-1}}r^{\frac {1-n} {p-1}(p-2)}  b^2 \le C(\ez_0,n,p)r^{-\frac p {p-1}}\bigl(r^{\frac {1-n} {p-1}}+\ez b\bigr)^{p-2}b^2;
	\end{multline*}
- if $b\ge \ez^{-1} r^{\frac {1-n} {p-1}}$ then we apply the inequality
	$\ez a\le \zeta \bigl(1+r^{\frac p {p-1}}\bigr)^{1-\frac n p}\le 2\zeta r^{\frac {p-n} {p-1}}$ to get
	$$a^{2-p}b^p r^{\frac{p-n}{p-1}(p^*-2)+p}\le  2^{2-p}r^{-\frac{p }{p-1}}  \zeta^{2-p} \ez^{p-2}b^p \le  C(\ez_0,n,p) r^{-\frac{p }{p-1}}\bigl(r^{\frac {1-n} {p-1}}+\ez b\bigr)^{p-2}b^2.$$
This proves \eqref{young r large} whenever $\zeta^p\le \frac {\ez_0}{3}$, concluding the proof of \eqref{inter}.
\end{proof}

\end{document}